\documentclass[12pt]{article}

\usepackage{a4,graphicx,amsmath,amsfonts,amssymb,mathrsfs,bbm,color}
\usepackage{amsthm}
\usepackage{caption,subcaption,mathalfa}
\usepackage{datetime}
\usepackage[EFvoltages]{circuitikz}
\ctikzset{bipoles/thickness=1.2}

\newcommand{\R}{\mathbbm{R}}

\newcommand{\N}{\mathbbm{N}}
\newcommand{\ltworandom}{\mathcal{L}^2(\mathcal{M},\rho)}
\newcommand{\htwonorm}{\mathcal{H}_2}
\newcommand{\opG}{\mathscr{G}}
\newcommand{\opP}{\mathscr{P}}
\newcommand{\cS}{\mathcal{S}}
\newcommand{\dd}{{\operatorname{d}}}
\DeclareMathOperator{\diag}{diag}
\DeclareMathOperator{\Var}{Var}

\newtheorem{definition}{Definition}[section]
\newtheorem{theorem}[definition]{Theorem}

\newtheorem{lemma}[definition]{Lemma}

\newtheorem{remark}[definition]{Remark}

\begin{document}

%%%%%%%%%%%%%%%%%%%%%%%%%%%%%%%%%%%%%%%%%%%%%%%%%%%%%%%%%%%%%%%%%%%%%%%%%%%%%
%%%                          Title                                        %%%
%%%%%%%%%%%%%%%%%%%%%%%%%%%%%%%%%%%%%%%%%%%%%%%%%%%%%%%%%%%%%%%%%%%%%%%%%%%%%

\begin{center}
  {\bf \Large Stochastic Galerkin method and \\[0.1ex] 
        port-Hamiltonian form for linear \\[0.9ex]
        first-order ordinary differential equations}

\vspace{10mm}

{\large Roland~Pulch and Olivier S{\`e}te} \\[2ex]
{\small Institute of Mathematics and Computer Science, 
Universit\"at Greifswald, \\
Walther-Rathenau-Str.~47, 17489 Greifswald, Germany. \\
Email: {\tt [roland.pulch,olivier.sete]@uni-greifswald.de}}

\bigskip

% \today

\end{center}

\bigskip

%%%%%%%%%%%%%%%%%%%%%%%%%%%%%%%%%%%%%%%%%%%%%%%%%%%%%%%%%%%%%%%%%%%%%%%%%%%%%
%%%                         Abstract                                      %%%
%%%%%%%%%%%%%%%%%%%%%%%%%%%%%%%%%%%%%%%%%%%%%%%%%%%%%%%%%%%%%%%%%%%%%%%%%%%%%

\begin{center}
{\bf Abstract}

\begin{tabular}{p{13cm}}
We consider linear first-order systems of ordinary differential equations 
(ODEs) in port-Hamiltonian (pH) form. 
Physical parameters are remodelled as random variables to conduct 
an uncertainty quantification.  
A stochastic Galerkin projection yields a larger deterministic system 
of ODEs, which does not exhibit a pH form in general. 
We apply transformations of the original systems such that the 
stochastic Galerkin projection becomes structure-pre\-serv\-ing. 
Furthermore, we investigate meaning and properties of the 
Hamiltonian function belonging to the stochastic Galerkin system. 
A large number of random variables implies a high-di\-men\-sional 
stochastic Galerkin system, which suggests itself to apply 
model order reduction (MOR) generating a low-dimensional system of ODEs. 
We discuss struc\-ture pre\-servation in projection-based MOR, 
where the smaller systems of ODEs feature pH form again. 
Results of numerical computations are presented using two test examples. 

\bigskip

MSC-2020 classification: 65L05, 34F05, 93D30

\bigskip

\end{tabular}
\end{center}

\clearpage

%%%%%%%%%%%%%%%%%%%%%%%%%%%%%%%%%%%%%%%%%%%%%%%%%%%%%%%%%%%%%%%%%%%%%%%%%%%%%
%%%                        Introduction                                   %%%
%%%%%%%%%%%%%%%%%%%%%%%%%%%%%%%%%%%%%%%%%%%%%%%%%%%%%%%%%%%%%%%%%%%%%%%%%%%%%

\section{Introduction}
Modelling and analysis of physical problems may yield 
systems of differential equations, which can be arranged in 
port-Hamiltonian (pH) form, see~\cite{jacob-zwart,schaft-jeltsema}. 
Such a model is strongly related to the energy within the system. 
Important physical properties are included in a pH system 
like energy storage, energy dissipation, and others. 
Each pH system has a Hamiltonian function, 
which describes an internal energy. 
Thus a pH formulation is stable as well as passive 
due to its special form. 

Dynamical systems typically include physical parameters, 
which are often affected by uncertainties due to modelling errors, 
measurement errors, or others. 
In uncertainty quantification (UQ), a common approach consists in 
substituting some parameters by random variables to describe 
their variability, see~\cite{sullivan-book,xiu-book}. 
The ran\-dom-dependent solution is expanded in a series of 
the so-called polynomial chaos (PC). 
This stochastic model can be solved by stochastic collocation methods 
or stochastic Galerkin (SG) methods. 
In stochastic collocation techniques, the original dynamical systems 
are solved (many times) for different realisations of the parameters. 
The SG approach yields a larger deterministic system 
of differential equations, which has to be solved once, 
see~\cite{pulch-xiu,pulch2018}.

Considering linear second-order systems of ordinary differential 
equations (ODEs), the symmetry and positive (semi-)definiteness of 
matrices implies a pH formulation, cf.~\cite{beattie-mehrmann}. 
In~\cite{pulch2023}, the SG approach was examined 
for linear second-order systems of ODEs, where the focus was on 
properties of a pH formulation. 
The SG projection preserves the structure in this case. 

In this article, we investigate linear first-order systems of ODEs 
in pH form. 
Now the SG projection is not structure-preserving 
in general. 
We consider two parameter-dependent transformations of the original 
system, where each transformation yields a linear pH system 
of a special type, see~\cite{huang-jiang-xu,morandin-etal}. 
It follows that the SG method produces a pH system 
of the same type. 
Furthermore, we examine the properties of the Hamiltonian function 
belonging to the SG system. 

An SG system is high-dimensional 
in the case of large numbers of random parameters. 
Hence we study model order reduction (MOR) of the large dynamical systems. 
There are efficient methods for MOR of general linear ODE systems, 
see~\cite{antoulas}. 
In addition, several previous works treated MOR of linear pH systems, 
see~\cite{gugercin-etal,ionescu-astolfi}, for example.
We investigate the structure-preservation of projection-based 
MOR techniques when applied to the high-dimensional pH systems 
from the SG projection. 
Finally, we present results of numerical computations, 
where the original pH systems are well-known test examples, 
cf.~\cite{ionescu-astolfi,schaft-jeltsema}.

The article is organised as follows. 
We introduce pH systems and associated transformations 
in Section~\ref{sec:problem}. 
The stochastic model is arranged in Section~\ref{sec:stoch-model}, 
where also polynomial expansions are outlined. 
In Section~\ref{sec:galerkin}, 
we investigate the SG method applied to pH systems. 
MOR of the SG systems is discussed 
in Section~\ref{sec:mor}. 
Finally, Section~\ref{sec:example} includes numerical results 
of the test examples.

%%%%%%%%%%%%%%%%%%%%%%%%%%%%%%%%%%%%%%%%%%%%%%%%%%%%%%%%%%%%%%%%%%%%%%%%%%%%%
%%%                       Port-Hamiltonian                                %%%
%%%%%%%%%%%%%%%%%%%%%%%%%%%%%%%%%%%%%%%%%%%%%%%%%%%%%%%%%%%%%%%%%%%%%%%%%%%%%

\section{Port-Hamiltonian Systems}
\label{sec:problem}

A general first-order linear dynamical system has the form 
\begin{align} \label{eq:ode}
\begin{split}
E \dot{x} & = A x + B u \\
y & = C x + D u 
\end{split}
\end{align}
with matrices $A,E \in \R^{n \times n}$, $B \in \R^{n \times n_{\rm in}}$, 
$C \in \R^{n_{\rm out} \times n}$, and 
$D \in \R^{n_{\rm out} \times n_{\rm in}}$. 
There are state variables $x : [0,\infty) \rightarrow \R^n$, 
inputs $u : [0,\infty) \rightarrow \R^{n_{\rm in}}$, 
and outputs $y : [0,\infty) \rightarrow \R^{n_{\rm out}}$. 
%The system is asymptotically stable, if and only if all 
%eigenvalues satisfying $\det(\lambda E - A)=0$ have a 
%strictly negative real part.

We use the notation $\mathcal{S}_{\succ}^{n}$ for the set of real 
symmetric, positive definite matrices with dimension~$n \times n$, 
and, likewise, $\mathcal{S}_{\succeq}^{n}$ for positive semi-definite 
matrices. 
We define linear first-order pH systems as in~\cite{beattie-mehrmann}.

%%%%%%%%%%%%%%%%%%%%%%%%%%%%%%%%%%%%%%%%%%%%%%%%%%%%%%%%%%%%%%%%%%%%%%%%%%%%%
\begin{definition} \label{def:ph}
The form of a linear first-order port-Hamiltonian (pH) system is
\begin{align} \label{eq:ph}
\begin{split}
E \dot{x} & = (J-R) Q x + (B-P) u \\
y & = (B+P)^\top Q x + (S+N) u 
\end{split}
\end{align}
with matrices $E, J,R,Q \in \R^{n \times n}$, 
$B,P \in \R^{n \times m}$, $S,N \in \R^{m \times m}$ 
such that~$J$ and~$N$ are skew-symmetric, 
$E^\top Q \in \mathcal{S}_{\succeq}^n$, and 
\begin{equation} \label{eq:matrix-w} 
W = \begin{pmatrix} 
Q^\top R Q & Q^\top P \\
P^\top Q & S \\
\end{pmatrix}
\in \mathcal{S}_{\succeq}^{n+m} . 
\end{equation}
The associated Hamiltonian function is given by
\begin{equation} \label{eq:hamiltonian}
H(x) = \tfrac{1}{2} \, x^\top (E^\top Q) x 
\end{equation}
for $x \in \R^n$.
\end{definition}
%%%%%%%%%%%%%%%%%%%%%%%%%%%%%%%%%%%%%%%%%%%%%%%%%%%%%%%%%%%%%%%%%%%%%%%%%%%%%

We assume that the mass matrix~$E$ is non-singular. 
Consequently, the system~\eqref{eq:ph} consists of ODEs. 
The system represents differential-algebraic equations (DAEs), 
if the mass matrix is singular.

A pH system is Lyapunov stable due to this structure. 
Furthermore, a pH system is passive, and the 
Hamiltonian function~\eqref{eq:hamiltonian} represents 
a storage function characterising an internal energy, 
see~\cite{schaft-jeltsema,willems}.

Our aim is a structure preservation in the SG method. 
A struc\-ture-preserving projection is given in the case of 
$Q = I_n$ with the identity matrix~$I_n \in \R^{n \times n}$. 
Such a pH system exhibits the form
\begin{align} \label{eq:ph-trafo}
\begin{split}
\tilde{E} \dot{x} & = (\tilde{J} - \tilde{R}) x + (\tilde{B}-\tilde{P}) u \\
y & = (\tilde{B}+\tilde{P})^\top x + (S+N) u ,
\end{split}
\end{align}
where $\tilde{E} \in \mathcal{S}_{\succeq}^n$, 
$\tilde{R} \in \mathcal{S}_{\succeq}^n$, 
and $\tilde{J}$ is skew-symmetric.  
It holds that $\tilde{E} \in \mathcal{S}_{\succ}^n$ in the case of ODEs.
There are already two approaches known in the literature to 
achieve this property, cf.~\cite{huang-jiang-xu,morandin-etal}, 
for example.

%%%%%%%%%%%%%%%%%%%%%%%%%%%%%%%%%%%%%%%%%%%%%%%%%%%%%%%%%%%%%%%%%%%%%%%%%%%%%
\subsection{Basis Transformation}
\label{sec:trafo1}
Let $Q \in \mathcal{S}_{\succ}^n$. 
We consider a symmetric decomposition $Q = T T^\top$. 
The matrix~$T$ is non-singular.
For $n > 1$, there is an infinite number of such symmetric decompositions.
Popular choices are
\begin{enumerate}
\renewcommand{\labelenumi}{\roman{enumi})}
\item Cholesky decomposition: $T$ is a lower triangular matrix.

\item Matrix square root: 
Let $Q = S D S^\top$ be the eigendecomposition of~$Q$ 
with $D = \diag(\lambda_1,\ldots,\lambda_n)$ and orthogonal $S$. 
Then
$T = Q^{\frac12} = S D^{\frac12} S^\top$, 
where $D^{\frac12} = \diag(\sqrt{\lambda_1},\ldots,\sqrt{\lambda_n})$, 
satisfies $Q = T^2$ and $T \in \mathcal{S}_{\succ}^n$. 
\end{enumerate}
%%%%%%%%%%%%%%%%%%%%%%%%%%%%%%%%%%%%%%%%%%%%%%%%%%%%%%%%%%%%%%%%%%%%%%%%%%%%%
\begin{lemma} \label{lemma:trafo1}
Let a pH system~\eqref{eq:ph} be given and let $Q = T T^\top \in \cS_\succ^n$.
Using the basis transformation $\tilde{x} = T^\top x$, we define
\begin{equation*}
\tilde{E} = T^\top E T^{-\top}, \quad 
\tilde{J} = T^\top J T, \quad 
\tilde{R} = T^\top R T, \quad 
\tilde{B} = T^\top B, \quad 
\tilde{P} = T^\top P.
\end{equation*}
Then the system~\eqref{eq:ph} is equivalent to the pH system
\begin{align} \label{eq:ph-trafo-tilde}
\begin{split}
\tilde{E} \dot{\tilde{x}} & = (\tilde{J} - \tilde{R}) \tilde{x} + (\tilde{B}-\tilde{P}) u, \\
y & = (\tilde{B}+\tilde{P})^\top \tilde{x} + (S+N) u,
\end{split}
\end{align}
of the form~\eqref{eq:ph-trafo}.
The Hamiltonian functions $H$ of~\eqref{eq:ph} and $\tilde{H}$ 
of~\eqref{eq:ph-trafo-tilde} are related by
\begin{equation} \label{eqn:hamiltonian_basistrafo}
\tilde{H}(\tilde{x}) = H(x), \quad \tilde{x} = T^\top x, \quad x \in \R^n.
\end{equation}
\end{lemma}

\begin{proof}
Note that $\dot{\tilde{x}} = T^\top \dot{x}$,
because $T$ is constant.  
Since $T$ is non-singular, the first equation
in~\eqref{eq:ph} is equivalent to
\begin{equation*}
T^\top E T^{-\top} T^\top \dot{x} = T^\top (J-R) T T^\top x + T^\top (B-P) u,
\end{equation*}
i.e., to the first equation in~\eqref{eq:ph-trafo-tilde} 
for the state variables.  
The equivalence of the second equation in~\eqref{eq:ph} and~\eqref{eq:ph-trafo-tilde} for the outputs is immediate.
Next, we show that~\eqref{eq:ph-trafo-tilde} is pH.
Since $J$ is skew-symmetric, so is $\tilde{J}$.  Moreover
\begin{equation} \label{eqn:tildeEQ}
\tilde{E}^\top I_n = (T^\top E T^{-\top})^\top = T^{-1} E^\top T
= T^{-1} E^\top T T^\top T^{-T} = T^{-1} (E^\top Q) T^{-\top},
\end{equation}
hence $\tilde{E}^\top$ is symmetric positive (semi-)definite if and 
only if $E^\top Q$ is symmetric positive (semi-)definite.
Finally, it holds that
\begin{equation*}
\tilde{W}
= \begin{pmatrix} I_n^\top \tilde{R} I_n & I_n^\top
\tilde{P} \\ \tilde{P}^T I_n & S \end{pmatrix}
= \begin{pmatrix} T \\ & I_m \end{pmatrix}^{-1} W \begin{pmatrix} T \\ & I_m 
\end{pmatrix}^{- \top} \in \cS^{n+m}_\succeq,
\end{equation*}
since $W \in \cS^{n+m}_\succeq$ and $T$ is non-singular.
This shows that~\eqref{eq:ph-trafo-tilde} is a pH system.
The Hamiltonian function $\tilde{H}$ of~\eqref{eq:ph-trafo-tilde} satisfies,
using~\eqref{eqn:tildeEQ},
\begin{equation*}
\tilde{H}(\tilde{x})
= \tfrac{1}{2} \, \tilde{x}^\top \tilde{E}^\top I_n \tilde{x}
= \tfrac{1}{2} \, x^\top T (T^{-1} E^\top Q T^{-\top}) T^\top x
= \tfrac{1}{2} \, x^\top E^\top Q x = H(x).
\end{equation*}
This completes the proof.
\end{proof}

%%%%%%%%%%%%%%%%%%%%%%%%%%%%%%%%%%%%%%%%%%%%%%%%%%%%%%%%%%%%%%%%%%%%%%%%%%%%%

An explicit system of ODEs remains an explicit system 
under this change of basis. 

%%%%%%%%%%%%%%%%%%%%%%%%%%%%%%%%%%%%%%%%%%%%%%%%%%%%%%%%%%%%%%%%%%%%%%%%%%%%%
\subsection{Transformation in Image Space} 
\label{sec:trafo2}
Let $Q \in \R^{n \times n}$ just be non-singular. 
Multiplication of the dynamical part in~\eqref{eq:ph} with $Q^\top$ yields
\begin{align*} 
\begin{split}
Q^\top E \dot{x} & = (Q^\top J Q - Q^\top R Q) x + (Q^\top B- Q^\top P) u \\
y & = (Q^\top B + Q^\top P)^\top x + (S+N) u.
\end{split}
\end{align*}
%%%%%%%%%%%%%%%%%%%%%%%%%%%%%%%%%%%%%%%%%%%%%%%%%%%%%%%%%%%%%%%%%%%%%%%%%%%%%
\begin{lemma} \label{lemma:trafo2}
Let a pH system~\eqref{eq:ph} be given with non-singular matrix $Q$.
Define
\begin{equation*}
\tilde{E} = Q^\top E, \quad 
\tilde{J} = Q^\top J Q, \quad  
\tilde{R} = Q^\top R Q, \quad 
\tilde{B} = Q^\top B, \quad 
\tilde{P} = Q^\top P.
\end{equation*}
Then the pH systems~\eqref{eq:ph} and~\eqref{eq:ph-trafo} are equivalent.
% In particular, $\tilde{E} \in \mathcal{S}_{\succ}^n$.
% due to Definition~\ref{def:ph}. 
% Furthermore, $\tilde{J}$ is skew-symmetric 
% and $\tilde{R} \in \mathcal{S}_{\succeq}^n$.
The Hamiltonian functions $H$ of~\eqref{eq:ph} and $\tilde{H}$
of~\eqref{eq:ph-trafo} are equal, i.e.,
\begin{equation*}
\tilde{H}(x) = H(x), \quad x \in \R^n.
\end{equation*}
\end{lemma}

\begin{proof}
Multiplication of the first equation in~\eqref{eq:ph} yields the equivalence
of~\eqref{eq:ph} and~\eqref{eq:ph-trafo}, since $Q$ is non-singular.
Next, we show that~\eqref{eq:ph-trafo} has pH structure.
Obviously, $\tilde{J} = Q^\top J Q$ is again skew-symmetric,
since $J$ has this property.
Moreover, 
$\tilde{E}^\top I_n = (Q^\top E)^\top = E^\top Q \in \cS^n_\succeq$ and
% (even $\cS^n_\succ$ if $E$ is non-singular) 
\begin{equation*}
\tilde{W}
% \begin{pmatrix} \tilde{Q}^\top \tilde{R} \tilde{Q} & \tilde{Q}^\top \tilde{P}
% \\ \tilde{P}^\top \tilde{Q} & \tilde{S} \end{pmatrix}
= \begin{pmatrix} \tilde{R} & \tilde{P} \\ \tilde{P}^\top & S \end{pmatrix}
= W \in \cS^{n+m}_\succeq.
\end{equation*}
The Hamiltonian function of~\eqref{eq:ph-trafo} is
\begin{equation*}
\tilde{H}(x) = \tfrac{1}{2} \, x^\top \tilde{E}^\top I_n x
= \tfrac{1}{2} \, x^\top E^\top Q x = H(x)
\end{equation*}
for $x \in \R^n$. This completes the proof.
\end{proof}
%%%%%%%%%%%%%%%%%%%%%%%%%%%%%%%%%%%%%%%%%%%%%%%%%%%%%%%%%%%%%%%%%%%%%%%%%%%%%

This technique represents a basis transformation in the image space, 
whereas the state space remains unchanged.  
In the case of explicit systems of ODEs~\eqref{eq:ph}, a disadvantage is
that the transformed system~\eqref{eq:ph-trafo} consists of 
implicit ODEs.

%%%%%%%%%%%%%%%%%%%%%%%%%%%%%%%%%%%%%%%%%%%%%%%%%%%%%%%%%%%%%%%%%%%%%%%%%%%%%
%%%                       Stochastic Modelling                            %%%
%%%%%%%%%%%%%%%%%%%%%%%%%%%%%%%%%%%%%%%%%%%%%%%%%%%%%%%%%%%%%%%%%%%%%%%%%%%%%

%%%%%%%%%%%%%%%%%%%%%%%%%%%%%%%%%%%%%%%%%%%%%%%%%%%%%%%%%%%%%%%%%%%%%%%%%%%%%
\section{Stochastic Modelling}
\label{sec:stoch-model}
Now we address a variability in the parameters of the 
linear dynamical systems. 

%%%%%%%%%%%%%%%%%%%%%%%%%%%%%%%%%%%%%%%%%%%%%%%%%%%%%%%%%%%%%%%%%%%%%%%%%%%%%
\subsection{Parameter-dependent Systems} 
Typically, mathematical models include physical parameters and/or 
other parameters. 
We assume that the parameters $\mu \in \mathcal{M}$ with 
parameter domain $\mathcal{M} \subseteq \R^q$ are present in 
a linear pH system~\eqref{eq:ph}. 
Obviously, a transformed system~\eqref{eq:ph-trafo} with matrix 
$\tilde{Q} = I_n$ depends on the same parameters. 
Thus we consider a pH system of the form
\begin{align} \label{eq:ph-par}
\begin{split}
E(\mu) \dot{x} & = (J(\mu)-R(\mu)) x + (B(\mu)-P(\mu)) u \\
y & = (B(\mu)+P(\mu))^\top x + (S(\mu)+N(\mu)) u .
\end{split}
\end{align}
The state variables $x : [0,\infty) \times \mathcal{M} \rightarrow \R^n$ 
depend on time as well as the parameters. 
Often the parameters are affected by uncertainties. 
Hence a variability of the parameters has to be taken into account. 

%%%%%%%%%%%%%%%%%%%%%%%%%%%%%%%%%%%%%%%%%%%%%%%%%%%%%%%%%%%%%%%%%%%%%%%%%%%%%
\subsection{Random Variables and Function Spaces}
The parameters~$\mu \in \mathcal{M}$ of the linear dynamical 
system~\eqref{eq:ph-par} are substituted by independent 
random variables on a probability space $(\Omega,\mathcal{A},P)$. 
Traditional probability distributions can be employed for each 
random variable like uniform distribution, beta distribution, 
gamma distribution, Gaussian distribution, etc. 
Let $\rho : \mathcal{M} \rightarrow \R$ be a joint 
probability density function. 
The expected value of a measurable function $f : \mathcal{M} \rightarrow \R$ 
is defined by
\begin{equation} \label{eq:expected-value}
\mathbb{E}[f] = \int_{\Omega} f(\mu(\omega)) \; \mathrm{d}P(\omega)  
= \int_{\mathcal{M}} f(\mu) \rho(\mu) \; \mathrm{d}\mu .
\end{equation}
The inner product of two measurable functions $f$ and $g$ reads as
\begin{equation} \label{eq:inner-product}
\langle f , g \rangle = 
\int_{\mathcal{M}} f(\mu) g(\mu) \rho(\mu) \; \mathrm{d}\mu . 
\end{equation} 
This inner product is well-defined on the space of 
square integrable functions
\begin{equation*}
\ltworandom = \left\{ f : \mathcal{M} \rightarrow \R \; : \; 
f \; \mbox{measurable and} \; \mathbb{E}[f^2] < \infty \right\} .
\end{equation*}
The induced norm is given by 
$\| f \|_{\ltworandom} = \sqrt{\langle f , f \rangle}$. 

%%%%%%%%%%%%%%%%%%%%%%%%%%%%%%%%%%%%%%%%%%%%%%%%%%%%%%%%%%%%%%%%%%%%%%%%%%%%%
\subsection{Polynomial Chaos Expansions}
Let $(\Phi_i)_{i \in \N}$ be an orthonormal basis with respect to the 
inner product~\eqref{eq:inner-product}, which consists of 
polynomials $\Phi_i : \mathcal{M} \rightarrow \R$. 
We assume that $\Phi_1 \equiv 1$ is the basis polynomial of degree zero. 
A function $f \in \ltworandom$ exhibits the polynomial chaos (PC) expansion 
\begin{equation} \label{eq:pc-f}
f(\mu) = \sum_{i=1}^{\infty} f_i \Phi_i(\mu)
\end{equation} 
with real-valued coefficients
\begin{equation*}
f_i = \langle f , \Phi_i \rangle .
\end{equation*}
The series~\eqref{eq:pc-f} converges in the $\ltworandom$-norm.
We recover the expected value as well as the variance by
\begin{equation*}
\mathbb{E} [f] = f_1 \quad \mbox{and} \quad 
\Var[f] = \sum_{i=2}^\infty f_i^2.
\end{equation*}
In the case of time-dependent functions $f(t,\mu)$ for $t \in I$ 
with $I \subseteq \R$, 
the PC expansion is used pointwise for each $t \in I$. 
We also apply the PC expansion to vector-valued functions or 
matrix-valued functions by considering each component separately. 

Concerning a pH system of ODEs, we use the PC expansions for 
the state variables, the inputs, and the outputs 
\begin{equation} \label{eq:pce} 
x(t,\mu) = \sum_{i=1}^{\infty} v_i(t) \Phi_i(\mu), \;
u(t,\mu) = \sum_{i=1}^{\infty} u_i(t) \Phi_i(\mu), \; 
y(t,\mu) = \sum_{i=1}^{\infty} w_i(t) \Phi_i(\mu) .  
\end{equation} 
The coefficients $u_i \in \R^m$ are known from the inputs, 
whereas the coefficients~$v_i \in \R^n$ and~$w_i \in \R^m$ 
are unknown a priori. 
In numerical methods, 
we truncate the PC expansions~\eqref{eq:pce} to $s$~stochastic modes
\begin{equation} \label{eq:pc-truncated} 
x^{(s)}(t,\mu) = \sum_{i=1}^s v_i(t) \Phi_i(\mu), \;
u^{(s)}(t,\mu) = \sum_{i=1}^s u_i(t) \Phi_i(\mu), \; 
y^{(s)}(t,\mu) = \sum_{i=1}^s w_i(t) \Phi_i(\mu) .  
\end{equation} 
Typically, all basis polynomials up to some total degree~$d$ are included. 
It follows that the number of terms becomes 
$s = \frac{(d+q)!}{d!q!}$, see~\cite{xiu-book}. 

\begin{remark}
In many applications, the inputs are independent, 
and thus do not depend on the parameters. 
This case can be described by choosing $u_1=u$ and $u_i=0$ for $i>1$ 
in~\eqref{eq:pce} or~\eqref{eq:pc-truncated}. 
However, we have to keep the sum including $s$~PC coefficients, 
since a pH system requires as many inputs as outputs. 
\end{remark}

%%%%%%%%%%%%%%%%%%%%%%%%%%%%%%%%%%%%%%%%%%%%%%%%%%%%%%%%%%%%%%%%%%%%%%%%%%%%%
%%%                   Stochastic Galerkin Method                          %%%
%%%%%%%%%%%%%%%%%%%%%%%%%%%%%%%%%%%%%%%%%%%%%%%%%%%%%%%%%%%%%%%%%%%%%%%%%%%%%

\section{Stochastic Galerkin Method}
\label{sec:galerkin}
There are SG methods and stochastic collocation methods 
to determine unknown coefficients of a PC expansion, 
see~\cite{sullivan-book,xiu-book}. 
We examine SG systems, whereas just the original 
systems of ODEs are solved for different realisations of the parameters 
in a stochastic collocation technique. 

%%%%%%%%%%%%%%%%%%%%%%%%%%%%%%%%%%%%%%%%%%%%%%%%%%%%%%%%%%%%%%%%%%%%%%%%%%%%%
\subsection{Stochastic Galerkin Projection}
If the truncated PC expansions~\eqref{eq:pc-truncated} are 
inserted into the pH system, then a residual is generated. 
The Galerkin approach requires that this residual is orthogonal 
to the subspace spanned by the basis polynomials 
$\{ \Phi_1,\ldots,\Phi_s \}$ with respect to the 
inner product~\eqref{eq:inner-product}. 
Basic calculations yield a larger deterministic first-order system 
of linear ODEs. 

We define an additional function space as well as an associated operator, 
as in~\cite{pulch2023}, 
to describe and analyse the SG method.

%%%%%%%%%%%%%%%%%%%%%%%%%%%%%%%%%%%%%%%%%%%%%%%%%%%%%%%%%%%%%%%%%%%%%%%%%%%%%
\begin{definition} \label{def:function-space}
The set of all measurable functions 
$A : \mathcal{M} \rightarrow \R^{n \times m}$, $A = (a_{k\ell})$,  
such that the expected values, cf.~\eqref{eq:expected-value},
\begin{equation} \label{eq:aij}
\hat{a}_{ijk\ell} = \mathbb{E} \left[ a_{k\ell} \Phi_i \Phi_j \right] = 
\int_{\mathcal{M}} a_{k\ell}(\mu) \Phi_i(\mu) \Phi_j(\mu) \rho(\mu) \; 
\mathrm{d}\mu 
\end{equation}
are finite for all $i,j \in \N$ and $1 \le k \le n$, $1 \le \ell \le m$, 
is denoted by $\mathcal{F}^{n,m}$. 
The SG projection of $A \in \mathcal{F}_{n,m}$ 
with $s$ modes is 
\begin{equation} \label{eq:projection-matrix} 
\hat{A} = \opG_s(A) 
% = (\hat{a}_{ij})_{i,j=1,\ldots,s} \in \R^{ns \times ms} . 
\end{equation}
with $\hat{A} = (\hat{A}_{ij})_{i,j = 1,\ldots,s}$ consisting of 
the submatrices
$\hat{A}_{ij} = (\hat{a}_{ijk\ell})_{k, \ell} \in \R^{n \times m}$.
\end{definition}
%%%%%%%%%%%%%%%%%%%%%%%%%%%%%%%%%%%%%%%%%%%%%%%%%%%%%%%%%%%%%%%%%%%%%%%%%%%%%

We obtain an operator 
$\opG_s : \mathcal{F}_{n,m} \rightarrow \R^{ns \times ms}$ 
for each integer~$s \geq 1$. 
The operator~\eqref{eq:projection-matrix} is linear, 
because it holds that
\begin{equation} \label{eq:operator-linear} 
\opG_s( \alpha A + \beta B ) = 
\alpha \opG_s (A) + \beta \opG_s (B) 
\end{equation}
for $A,B \in \mathcal{F}_{n,m}$ and $\alpha,\beta \in \R$. 
However, this operator is not multiplicative, i.e., 
\begin{equation} \label{eq:not-multiplicative} 
\opG_s ( A B ) \neq \opG_s(A) \opG_s(B) 
\end{equation}
for $A \in \mathcal{F}_{n,m}$ and $B \in \mathcal{F}_{m,k}$ in general. 

The SG system associated to the 
pH system~\eqref{eq:ph-par} has the form
\begin{align} \label{eq:galerkin}
\begin{split}
\hat{E} \dot{\hat{v}} & = 
(\hat{J} - \hat{R}) \hat{v} + (\hat{B}+\hat{P}) \hat{u} \\
\hat{y} & = (\hat{B}-\hat{P})^\top \hat{v} + (\hat{S} + \hat{N}) \hat{u} .
\end{split}
\end{align}
The inputs and outputs are $\hat{u} = (u_1^\top,\ldots,u_s^\top)^\top$ 
and $\hat{w} = (\hat{w}_1^\top,\ldots,\hat{w}_s^\top)^\top$, 
respectively.
The unknown state variables are 
$\hat{v} = (\hat{v}_1^\top,\ldots,\hat{v}_s^\top)^\top$. 
We obtain approximations $\hat{v}_i \approx v_i$ and 
$\hat{w}_i \approx w_i$ of the exact coefficients in the 
PC expansions~\eqref{eq:pce}.
It holds that $\hat{E} = \opG_s(E)$, $\hat{J} = \opG_s(J)$, etc. 

%%%%%%%%%%%%%%%%%%%%%%%%%%%%%%%%%%%%%%%%%%%%%%%%%%%%%%%%%%%%%%%%%%%%%%%%%%%%%
\subsection{Structure-Preservation}
Now the question is if the SG system~\eqref{eq:galerkin}
is again a pH system as in Definition~\ref{def:ph}.

%%%%%%%%%%%%%%%%%%%%%%%%%%%%%%%%%%%%%%%%%%%%%%%%%%%%%%%%%%%%%%%%%%%%%%%%%%%%%
\begin{lemma} \label{lemma:definite}
Let $A \in \mathcal{F}^{n,n}$. 
If $A(\mu) \in \mathcal{S}_{\succ}^n$ for almost all $\mu \in \mathcal{M}$, 
then the SG projection~\eqref{eq:projection-matrix} 
satisfies $\hat{A} \in \mathcal{S}_{\succ}^{ns}$. 
Likewise, the property $A(\mu) \in \mathcal{S}_{\succeq}^n$ for almost all $\mu \in \mathcal{M}$ 
implies $\hat{A} \in \mathcal{S}_{\succeq}^{ns}$. 
If $A(\mu)$ is skew-symmetric for almost all~$\mu \in \mathcal{M}$, 
then $\hat{A}$ is also skew-symmetric.
\end{lemma}
%%%%%%%%%%%%%%%%%%%%%%%%%%%%%%%%%%%%%%%%%%%%%%%%%%%%%%%%%%%%%%%%%%%%%%%%%%%%%

The proof of the positive definite case can be found in~\cite{pulch-jmi}. 
Likewise, we obtain the result in the positive semi-definite case.

As an example, we show the preservation of symmetry as well as 
skew-symmetry. 
Let $A \in \mathcal{F}_{n,n}$ with $A^\top = A$ or $A^\top = - A$.  
The property $A^\top = \pm A$ implies $a_{k\ell} = \pm a_{\ell k}$ 
for $k,\ell = 1,\ldots,n$.
In view of~\eqref{eq:aij}, it follows that
\begin{equation*}
\hat{a}_{ijk\ell} = % \hat{a}_{jik\ell} 
\mathbb{E} \left[ a_{k\ell} \Phi_i \Phi_j \right] = 
\mathbb{E} \left[ \pm a_{\ell k} \Phi_i \Phi_j \right] =
\pm \hat{a}_{ji\ell k}
\end{equation*}
for all $i,j \in \N$ and $k,\ell = 1,\ldots,n$. 
Hence we obtain $\hat{A}^\top = \hat{A}$ or $\hat{A}^\top = -\hat{A}$. 

%%%%%%%%%%%%%%%%%%%%%%%%%%%%%%%%%%%%%%%%%%%%%%%%%%%%%%%%%%%%%%%%%%%%%%%%%%%%%
\begin{theorem} \label{thm:galerkin-structure}
Let a first-order pH system be given in the form~\eqref{eq:ph-par} 
with random variables $\mu : \Omega \rightarrow \mathcal{M}$. 
The SG projection yields a system of 
ODEs~\eqref{eq:galerkin}, which has a pH form again.
\end{theorem}
%%%%%%%%%%%%%%%%%%%%%%%%%%%%%%%%%%%%%%%%%%%%%%%%%%%%%%%%%%%%%%%%%%%%%%%%%%%%%

\begin{proof}
Lemma~\ref{lemma:definite} shows that the matrices in the 
SG system~\eqref{eq:galerkin} have the required properties 
of Definition~\ref{def:ph}: 
$\hat{E} \in \mathcal{S}_{\succ}^{ns}$, 
$\hat{R} \in \mathcal{S}_{\succeq}^{ns}$, 
$\hat{S} \in \mathcal{S}_{\succ}^{ms}$, and 
$\hat{J},\hat{N}$ are skew-symmetric. 
It holds that $\hat{Q} = I_{ns}$. 
The matrix $W$ in~\eqref{eq:matrix-w} associated to the pH system~\eqref{eq:ph} 
is symmetric and positive semi-definite for almost all 
$\mu \in \mathcal{M}$. 
Again Lemma~\ref{lemma:definite} implies that the stochastic 
Galerkin projection $\hat{W} = \opG_s(W)$ is symmetric and 
positive semi-definite. 
There is a permutation matrix $\Pi$ such that
\begin{equation*}
\Pi^\top \hat{W} \, \Pi = 
\begin{pmatrix}
\hat{R} & \hat{P} \\
\hat{P}^\top & \hat{S} \\
\end{pmatrix},
\end{equation*}
which is the matrix~\eqref{eq:matrix-w} 
belonging to the SG system~\eqref{eq:galerkin}. 
It follows that the symmetric matrix $\Pi^\top \hat{W} \, \Pi$ is 
also positive semi-definite. 
\end{proof}

%%%%%%%%%%%%%%%%%%%%%%%%%%%%%%%%%%%%%%%%%%%%%%%%%%%%%%%%%%%%%%%%%%%%%%%%%%%%%
\begin{remark}
The SG method applied to a pH system of the 
general form~\eqref{eq:ph} is not structure-preserving, 
since the SG projection $\opG_s$ is not 
multiplicative. 
Due to~\eqref{eq:operator-linear}, \eqref{eq:not-multiplicative}, 
it holds that, in general,
\begin{equation*}
\opG_s((J-R) Q) = \opG_s (JQ) - \opG_s(RQ)
\neq (\opG_s(J) - \opG_s(R)) \opG_s(Q).
\end{equation*}
Hence an alternative form like~\eqref{eq:ph-par} with $Q=I_n$ is required. 
\end{remark}
%%%%%%%%%%%%%%%%%%%%%%%%%%%%%%%%%%%%%%%%%%%%%%%%%%%%%%%%%%%%%%%%%%%%%%%%%%%%%

%%%%%%%%%%%%%%%%%%%%%%%%%%%%%%%%%%%%%%%%%%%%%%%%%%%%%%%%%%%%%%%%%%%%%%%%%%%%%
\subsection{Hamiltonian Function of Stochastic Galerkin System}
\label{sec:galerkin-hamiltonian}
The SG system~\eqref{eq:galerkin} satisfies the 
properties of a pH formulation. 
The associated Hamiltonian function is
\begin{equation} \label{eq:galerkin-hamiltonian}
\hat{H}(\hat{v}) = \tfrac12 \, \hat{v}^\top \hat{E} \hat{v}
\end{equation}
with the symmetric, positive definite mass matrix~$\hat{E}$, 
since it holds that $\hat{Q} = I_{ns}$.

%%%%%%%%%%%%%%%%%%%%%%%%%%%%%%%%%%%%%%%%%%%%%%%%%%%%%%%%%%%%%%%%%%%%%%%%%%%%
\begin{theorem} \label{thm:hamiltonian}
The Hamiltonian function of the SG system represents 
an approximation of the expected value of the Hamiltonian function 
for the parameter-dependent pH systems.
More precisely, the Hamiltonians $H(\cdot, \mu)$ of~\eqref{eq:ph-par}
and $\hat{H}$ of~\eqref{eq:galerkin-hamiltonian} are related by
\begin{equation*}
\hat{H}(\hat{v}) = \mathbb{E} \biggl[ H \biggl( \sum_{j=1}^s \hat{v}_j \Phi_j(\mu), \mu \biggr) \biggr]
% \quad \hat{v} = (\hat{v}_1^\top, \ldots, \hat{v}_s^\top)^\top \in \R^{sn}.
\end{equation*}
with $\hat{v} = (\hat{v}_1^\top, \ldots, \hat{v}_s^\top)^\top \in \R^{sn}$.
\end{theorem}
%%%%%%%%%%%%%%%%%%%%%%%%%%%%%%%%%%%%%%%%%%%%%%%%%%%%%%%%%%%%%%%%%%%%%%%%%%%%

\begin{proof}
The Hamiltonian function of~\eqref{eq:galerkin-hamiltonian} satisfies
\begin{align*}
\hat{H}(\hat{v}) &= \tfrac{1}{2} \, \hat{v}^\top \hat{E}^\top I_{sn} \hat{v}
= \tfrac{1}{2} \sum_{i,j=1}^s \hat{v}_i^\top \hat{E}_{ij} \hat{v}_j \\
&= \int_\mathcal{M} \tfrac{1}{2} \sum_{i,j=1}^s \hat{v}_i^\top E(\mu) 
\hat{v}_j \Phi_i(\mu) \Phi_j(\mu) \rho(\mu) \, \dd \mu \\
&= \int_\mathcal{M} \tfrac{1}{2} \biggl( \sum_{i=1}^s \hat{v}_i \Phi_i(\mu) \biggr)^\top E(\mu)^\top \biggl( \sum_{j=1}^s \hat{v}_j \Phi_j(\mu) \biggr) \rho(\mu) \, \dd \mu \\
&= \mathbb{E} \biggl[ H \biggl( \sum_{j=1}^s \hat{v}_j \Phi_j(\mu), \mu \biggr) \biggr],
\end{align*}
as claimed.
\end{proof}

Let $\hat{v}^{(d)}$ be the solution of an SG system 
considering all basis polynomials up to total degree~$d$. 
If the SG method is convergent 
($x^{(s)}$ in~\eqref{eq:pc-truncated} converges to $x$), 
then the Hamiltonian function 
$\hat{H}(\hat{v}^{(d)}(t))$ converges to $\mathbb{E}(H(x(t,\mu),\mu))$ 
for each $t \ge 0$ in the case of $d \rightarrow \infty$.

Each pH system is passive with respect to its Hamiltonian function 
as storage function. 
In the case of an SG system~\eqref{eq:galerkin}, 
we obtain with the Hamiltonian function~\eqref{eq:galerkin-hamiltonian}
\begin{equation*}
{\textstyle \frac{\rm d}{{\rm d}t}} \, \hat{H}(\hat{v}(t)) \le 
\hat{w}(t)^\top \hat{u} (t) = 
\sum_{i=1}^s \hat{w}_i(t)^\top {u}_i(t)
\end{equation*}
for all $t \ge 0$.
Often the inputs do not depend on the parameters in the 
system~\eqref{eq:ph-par}. 
It follows that $\hat{u}=(u^\top,0,\ldots,0)^\top$ with the 
inputs~$u$ from~\eqref{eq:ph-par}. 
Thus the bound simplifies to
\begin{equation*}
{\textstyle \frac{\rm d}{{\rm d}t}} \, \hat{H}(\hat{v}(t)) \le 
\hat{w}_1(t)^\top u(t) ,
\end{equation*}
where $\hat{w}_1(t) \approx \mathbb{E}[y(t,\mu)]$ generates 
an approximation of the expected value of the outputs 
in~\eqref{eq:ph-par}. 

In the SG system, there are $s$~outputs 
associated to the stochastic modes of the problem. 
The first mode corresponds to the expected value. 
Theorem~\ref{thm:hamiltonian} yields an information on the 
expected value of the Hamiltonian function in~\eqref{eq:ph-par}. 
An obvious question is if the higher stochastic modes of the 
Hamiltonian function can also be found in the SG system. 
However, each pH system includes just a single scalar Hamiltonian function. 
Let a solution of an initial value problem of the 
SG system~\eqref{eq:galerkin} be predetermined. 
It is straightforward to show that an approximation of the higher 
stochastic modes reads as 
$$ H_k(\hat{v}(t)) = \tfrac12 \, \hat{v}(t)^\top \hat{H}_k \hat{v}(t) $$
for $k=2,\ldots,s$
including the constant matrices 
$$ \hat{H}_k = \opG ( \tilde{E} \Phi_k ) \in \R^{ns \times ns} , $$
where the operator from Definition~\ref{def:function-space} is applied. 
It holds that $\hat{H}_1$ coincides with~\eqref{eq:galerkin-hamiltonian}.
All matrices $\hat{H}_k$ are symmetric. 
Yet the matrices are not positive definite or semi-definite for $k \ge 2$, 
because the basis polynomials~$\Phi_k$ exhibit both positive and 
negative values in the domain~$\mathcal{M}$. 

%%%%%%%%%%%%%%%%%%%%%%%%%%%%%%%%%%%%%%%%%%%%%%%%%%%%%%%%%%%%%%%%%%%%%%%%%%%%%
%%%                      Model Order Reduction                            %%%
%%%%%%%%%%%%%%%%%%%%%%%%%%%%%%%%%%%%%%%%%%%%%%%%%%%%%%%%%%%%%%%%%%%%%%%%%%%%%

\section{Model Order Reduction}
\label{sec:mor} 
If the number of random parameters is large, then an SG 
system~\eqref{eq:galerkin} has a high dimension. 
Hence an MOR is advantageous to decrease the complexity of the problem. 

%%%%%%%%%%%%%%%%%%%%%%%%%%%%%%%%%%%%%%%%%%%%%%%%%%%%%%%%%%%%%%%%%%%%%%%%%%%%%
\subsection{Projection-based MOR}
A full-order model (FOM) is converted into a 
reduced-order model (ROM). 
Often MOR is performed using projection matrices. 
We describe the projections by an operator. 

%%%%%%%%%%%%%%%%%%%%%%%%%%%%%%%%%%%%%%%%%%%%%%%%%%%%%%%%%%%%%%%%%%%%%%%%%%%%%
\begin{definition} \label{def:projection-operator}
Let $V,W \in \R^{n \times r}$ with $r < n$ be two projection matrices. 
The projection of a matrix $A \in \R^{n \times n}$ is defined as 
\begin{equation} \label{eq:projection} 
\bar{A} = \opP (A) = W^\top A V , 
\end{equation}
which generates the smaller matrix $\bar{A} \in \R^{r \times r}$.
\end{definition}
%%%%%%%%%%%%%%%%%%%%%%%%%%%%%%%%%%%%%%%%%%%%%%%%%%%%%%%%%%%%%%%%%%%%%%%%%%%%%

A full (column) rank is required in the projection matrices $V,W$. 
Often the biorthogonality condition $W^\top V = I_r$ with the 
identity matrix $I_r \in \R^{r \times r}$ is imposed on the 
projection matrices. 
The operator~\eqref{eq:projection} is linear, i.e., 
\begin{equation} \label{eq:p-linear} 
\opP (\alpha A + \beta B) = \alpha \opP(A) + \beta \opP(B) 
\end{equation}
for $A,B \in \R^{n \times n}$ and $\alpha,\beta \in \R$. 
However, the operator is not multiplicative, i.e., in general
\begin{equation} \label{eq:p-not-multiplicative} 
\opP(AB) \neq \opP(A) \opP(B) 
\end{equation}
for $A,B \in \R^{n \times n}$, 
because it holds that $V W^\top \neq I_n$ 
with the identity matrix $I_n \in \R^{n \times n}$.

A general system of ODEs~\eqref{eq:ode} is shrinked to an ROM
\begin{align} \label{eq:ode-reduced}
\begin{split}
\bar{E} \dot{\bar{x}} & = \bar{A} \bar{x} + \bar{B} u \\
\bar{y} & = \bar{C} \bar{x} + D u 
\end{split}
\end{align}
with matrices $\bar{A} = \opP(A)$, $\bar{E} = \opP(E)$, 
$\bar{B} = W^\top B$, and $\bar{C} = C V$.
A reduced system~\eqref{eq:ode-reduced} may be unstable 
even if the original system is Lyapunov-stable or asymptotically stable. 
There are stability-preserving MOR techniques like 
balanced truncation, see~\cite{antoulas}, for example.

If an SG system is given in the form of a general 
pH system as in Definition~\ref{def:ph}, then an ROM is not 
necessarily in pH form again due to the 
property~\eqref{eq:p-not-multiplicative}. 
Yet our SG system~\eqref{eq:galerkin} exhibits a 
form with $\hat{Q} = I_{mn}$. 
Since the operator~\eqref{eq:projection} is linear, 
see~\eqref{eq:p-linear}, an ROM has the form 
\begin{align} \label{eq:rom}
\begin{split}
\bar{E} \dot{\bar{v}} & = 
(\bar{J} - \bar{R}) \bar{v} + (\bar{B}+\bar{P}) \hat{u} \\
\bar{y} & = (\bar{B}'-\bar{P}')^\top \bar{v} + (\hat{S} + \hat{N}) \hat{u} .
\end{split}
\end{align}
with matrices $\bar{F} = \opP(\hat{F})$ for 
$\hat{F} \in \{ \hat{J}, \hat{R}, \hat{E} \}$ and
$\bar{G} = W^\top \hat{G}$, $\bar{G}' = V^\top \hat{G}$ 
for $\hat{G} \in \{ \hat{B}, \hat{P} \}$. 
We recognise that an ROM~\eqref{eq:rom} does not represent 
a pH system in general. 
%However, the question is if the shrinked matrices satisfy the required 
%properties in a pH system concerning symmetry and definiteness.

%%%%%%%%%%%%%%%%%%%%%%%%%%%%%%%%%%%%%%%%%%%%%%%%%%%%%%%%%%%%%%%%%%%%%%%%%%%%%
\subsection{Galerkin-type Projection}
A Galerkin-type projection-based MOR is defined by the property $W=V$, 
whereas a Petrov-Galerkin-type MOR is characterised by $W \neq V$. 
Thus a Galerkin-type method involves just a single projection matrix 
$V \in \R^{n \times r}$. 
This matrix is required to have orthonormal columns, i.e., $V^\top V = I_r$. 
Popular Galerkin-type MOR methods are the (one-side) Arnoldi algorithm 
and proper orthogonal decomposition (POD), see~\cite{antoulas}. 
Such MOR methods are structure-preserving in our problems 
with the SG systems~\eqref{eq:galerkin}.

%%%%%%%%%%%%%%%%%%%%%%%%%%%%%%%%%%%%%%%%%%%%%%%%%%%%%%%%%%%%%%%%%%%%%%%%%%%%%
\begin{theorem} \label{thm:galerkin-mor}
Let a pH system be given in the form~\eqref{eq:galerkin}. 
% with $\hat{Q}=I_{ns}$ 
A Galerkin-type MOR yields a reduced system~\eqref{eq:rom} in 
pH form again.
The Hamiltonian functions $\hat{H}$ of~\eqref{eq:galerkin-hamiltonian} and $\bar{H}$ of~\eqref{eq:rom} are related by
%\begin{equation} \label{eq:hamiltonian-reduced}
$$ \bar{H}(\bar{v}) = \hat{H}(V \bar{v}) 
\qquad \mbox{for} \;\; \bar{v} \in \R^r. $$
\end{theorem}
%%%%%%%%%%%%%%%%%%%%%%%%%%%%%%%%%%%%%%%%%%%%%%%%%%%%%%%%%%%%%%%%%%%%%%%%%%%%%

\begin{proof}
In a Galerkin-type MOR, it holds that $\bar{A} = \opP(A) = V^\top A V$ 
for a general matrix $A \in \R^{n \times n}$. 
It follows that the operator from Definition~\ref{def:projection-operator} 
preserves symmetry, skew-symmetry, and definiteness of the matrix,
since $V \in \R^{n \times r}$ has full (column) rank.
It holds that $\bar{A} \in \mathcal{S}_{\succeq}^r$ 
if $A \in \mathcal{S}_{\succeq}^n$. 
A semi-definite matrix may change into a definite matrix. 
Hence the matrices $\bar{E},\bar{J},\bar{R}$ satisfy the properties 
of a pH form in the system~\eqref{eq:rom}. 
No conditions are required for the matrices $\bar{B},\bar{P}$. 
The matrices $\hat{S},\hat{N}$ are the same as in the 
system~\eqref{eq:galerkin}.
Moreover,
\begin{equation*}
\bar{W} = \begin{pmatrix} \bar{R} & \bar{P} \\ \bar{P}^\top & \hat{S} \end{pmatrix}
= \begin{pmatrix} V^\top & 0 \\ 0 & I_m \end{pmatrix} \hat{W} 
\begin{pmatrix} V & 0 \\ 0 & I_m \end{pmatrix} % \in \cS^{r+m}_\succeq
\end{equation*}
is symmetric and positive semi-definite, since $\hat{W}$ has this property.
Finally, the Hamiltonian function satisfies
\begin{equation*}
\bar{H}(\bar{v}) = \tfrac{1}{2} \, \bar{v}^\top \bar{E}^\top I_r \bar{v}
= \tfrac{1}{2} \, \bar{v}^\top V^\top \hat{E}^\top V \bar{v}
= \tfrac{1}{2} \, (V \bar{v})^\top \hat{E}^\top (V \bar{v})
= \hat{H}(V \bar{v}),
\end{equation*}
as claimed above.
% \hfill $\Box$
\end{proof}

%%%%%%%%%%%%%%%%%%%%%%%%%%%%%%%%%%%%%%%%%%%%%%%%%%%%%%%%%%%%%%%%%%%%%%%%%%%%%
\begin{remark} 
Theorem~\ref{thm:galerkin-mor} is formulated for the 
SG system~\eqref{eq:galerkin}. 
Nevertheless, the statement is true for all pH system 
of the form~\eqref{eq:ph-trafo} with $\tilde{Q}=I_n$. 
\end{remark}
%%%%%%%%%%%%%%%%%%%%%%%%%%%%%%%%%%%%%%%%%%%%%%%%%%%%%%%%%%%%%%%%%%%%%%%%%%%%%

Galerkin-type MOR methods were applied to ODEs of pH form 
in~\cite{ionescu-astolfi,polyuga-schaft,wolf-etal}, for example.
There are also specific Petrov-Galerkin-type MOR methods or 
sophisticated modifications, 
which preserve the pH structure of ODE systems, 
see~\cite{gugercin-etal,huang-jiang-xu,xu-jiang}.
However, such methods are often constructed only for 
explicit systems of ODEs. 

%%%%%%%%%%%%%%%%%%%%%%%%%%%%%%%%%%%%%%%%%%%%%%%%%%%%%%%%%%%%%%%%%%%%%%%%%%%%%

%%%%%%%%%%%%%%%%%%%%%%%%%%%%%%%%%%%%%%%%%%%%%%%%%%%%%%%%%%%%%%%%%%%%%%%%%%%%%
%%%                        Illustrative examples                          %%%
%%%%%%%%%%%%%%%%%%%%%%%%%%%%%%%%%%%%%%%%%%%%%%%%%%%%%%%%%%%%%%%%%%%%%%%%%%%%%

\section{Illustrative Examples}
\label{sec:example}
We present numerical results for two test examples. 
The computations were performed using the software package 
MATLAB~\cite{matlab2020} on a
FUJITSU Esprimo P920 Intel(R) Core(TM) i7-9700 CPU with 3.00 GHz (8 cores) 
and operation system Microsoft Windows~10.

%%%%%%%%%%%%%%%%%%%%%%%%%%%%%%%%%%%%%%%%%%%%%%%%%%%%%%%%%%%%%%%%%%%%%%%%%%%%%
\subsection{DC Motor}
\label{sec:motor}
We consider a model of an electric motor from~\cite{schaft-jeltsema}, 
which is illustrated in Figure~\ref{fig:DCmotor}. 
The system of ODEs reads as
\begin{align} \label{eq:dcmotor}
\begin{split}
\begin{pmatrix} \dot{\varphi} \\ \dot{p} \\ \end{pmatrix} & = 
\begin{pmatrix} 
-R_{\rm m} & - K_{\rm m} \\ K_{\rm m} & - B_{\rm m} \\ 
\end{pmatrix} 
\begin{pmatrix} 
\frac{1}{L_{\rm m}} & 0 \\ 0 & \frac{1}{J_{\rm m}} \\ 
\end{pmatrix} 
\begin{pmatrix} \varphi \\ p \\ \end{pmatrix}
+ \begin{pmatrix} 1 \\ 0 \\ \end{pmatrix} V \\
I & = \begin{pmatrix} 1 & 0 \end{pmatrix} 
\begin{pmatrix} 
\frac{1}{L_{\rm m}} & 0 \\ 0 & \frac{1}{J_{\rm m}} \\ 
\end{pmatrix} 
\begin{pmatrix} \varphi \\ p \\ \end{pmatrix} 
\end{split}
\end{align} 
for the flux-linkage~$\varphi$ and the angular momentum~$p$. 
There are five (positive) physical parameters: 
an inductance~$L_{\rm m}$, a resistance~$R_{\rm m}$, 
a gyrator constant~$K_{\rm m}$, a friction~$B_{\rm m}$, 
and a rotational inertia~$J_{\rm m}$. 
The input voltage $u = V$ is supplied, 
while the output current is $y = I$. 
In this pH form, the matrices are
\begin{equation*}
J = \begin{pmatrix} 
0 & - K_{\rm m} \\ K_{\rm m} & 0 \\ 
\end{pmatrix} , \quad
R = \begin{pmatrix} 
R_{\rm m} & 0 \\ 0 & B_{\rm m} \\ 
\end{pmatrix} , \quad 
Q = \begin{pmatrix} 
\frac{1}{L_{\rm m}} & 0 \\ 0 & \frac{1}{J_{\rm m}} \\ 
\end{pmatrix} , \quad
B = \begin{pmatrix} 1 \\ 0 \\ \end{pmatrix}.
\end{equation*}
The matrix square root of $Q$ is just
\begin{equation} \label{eq:motor-square-root}
Q^{\frac12} =
\begin{pmatrix} 
\frac{1}{\sqrt{L_{\rm m}}} & 0 \\ 0 & \frac{1}{\sqrt{J_{\rm m}}}
\end{pmatrix}.
\end{equation}
As mean values, we use the physical parameters 
$L_{\rm m} = 0.001$, $R_{\rm m} = 0.01$, $K_{\rm m} = 10$, 
$B_{\rm m} = 1$, $J_{\rm m} = 1$. 
Figure~\ref{fig:bode-motor} depicts the Bode plot of the linear dynamical
system~\eqref{eq:dcmotor} for this constant choice of parameters. 
We observe a single resonance peak. 
However, the associated resonance frequency shifts depending on 
the parameters. 

%%%%%%%%%%%%%%%%%%%%%%%%%%%%%%%%%%%%%%%%%%%%%%%%%%%%%%%%%%%%%%%%%%%%%%%%%%%%%%
\begin{figure}
{\centering
\begin{circuitikz}
% motor: cylinder:
\draw[fill=lightgray!50](0,0) circle (0.5 and 1); % top
\draw[top color=gray!25,bottom color=black,middle color=gray!50] (-1.6,1) arc 
(90:270:0.5 and 1) -- ++(1.6,0) arc (-90:-270:0.5 and 1) -- cycle;
% rod
\draw[top color=white,bottom color=black!70] (0,3mm) arc (90:270:1.5mm and 
3mm)--++(1.5cm,0) arc (-90:-270:1.5mm and 3mm)-- cycle;
\draw (1.5cm,3mm) arc (90:-90:1.5mm and 3mm);
\node at (1.9, 0) {$J_{\rm m}$};
% arrows indicating rotation
\draw[-latex,thick] (0.9,-0.5) arc (300:50:0.3 and 0.6);
\draw[-latex,thick] (1.4,-0.5) arc (300:50:0.3 and 0.6);
\node at (-0.1, -0.5) {$B_{\rm m}$};
\node at (-2.5, 0) {$K_{\rm m}$};
% Attachements:
\draw[thick] (-1.1, 1) rectangle ++(0.2, 0.2);
\draw[thick] (-1.1, -1.2) rectangle ++(0.2, 0.2);
% top part of electrical circuit
\draw[thick] (-7, 2) node[below] {$+$} to[short, o-] (-6,2)
to [european resistor, l=$R_{\rm m}$] (-4,2)
to [cute inductor, l=$L_{\rm m}$] ++ (2,0)
-- (-1, 2) -- (-1, 1.2);
% bottom part of electrical circuit
\draw[thick] (-7, -2) node[above] {$-$} to[short, o-] (-4.5,-2) to [short, 
i<=$I$] (-3.5,-2) -- (-1, -2) -- (-1, -1.2);
\node at (-7,0) {$V$};
\end{circuitikz}

}
\caption{DC motor, see Section~\ref{sec:motor}.}
\label{fig:DCmotor}
\end{figure}
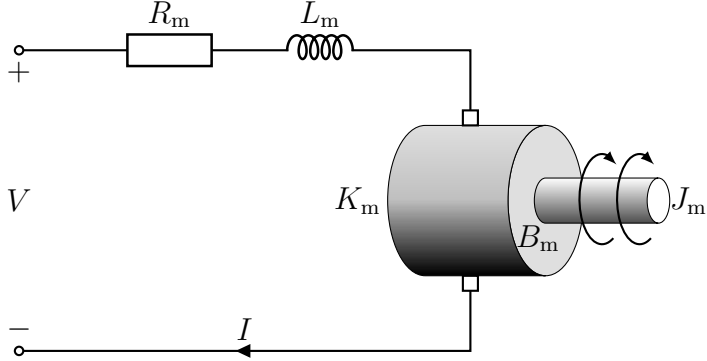
%%%%%%%%%%%%%%%%%%%%%%%%%%%%%%%%%%%%%%%%%%%%%%%%%%%%%%%%%%%%%%%%%%%%%%%%%%%%%%

%%% Figure: Bode Plot %%%%%%%%%%%%%%%%%%%%%%%%%%%%%%%%%%%%%%%%%%%%%%%%%%%%%%%
\begin{figure}
  \centering
  \begin{subfigure}[b]{0.45\textwidth}
  \centering
  \includegraphics[width=\textwidth]{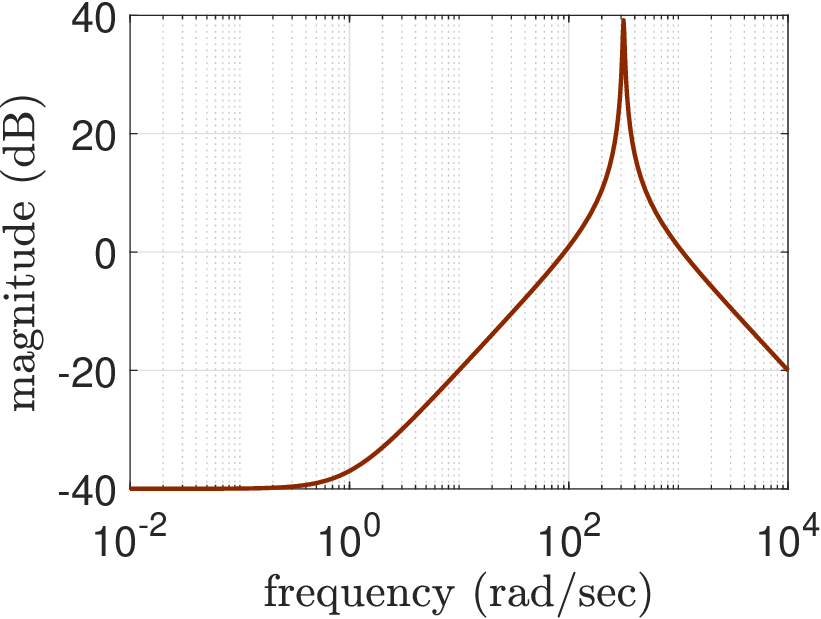}
  \caption{magnitude}
  \end{subfigure}
  \hspace{8mm}
  \begin{subfigure}[b]{0.45\textwidth}
  \centering
  \includegraphics[width=\textwidth]{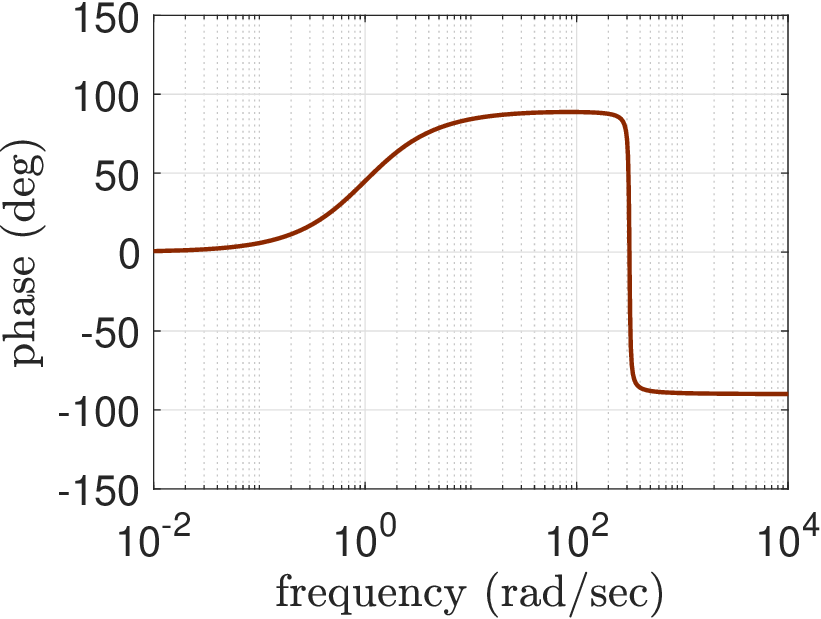}
  \caption{phase}
  \end{subfigure}
  \caption{Bode plot of DC motor model for deterministic 
  physical parameters.}
\label{fig:bode-motor}
\end{figure}
%%%%%%%%%%%%%%%%%%%%%%%%%%%%%%%%%%%%%%%%%%%%%%%%%%%%%%%%%%%%%%%%%%%%%%%%%%%%%

Now we replace the five physical parameters by independent random variables 
with uniform probability distributions varying symmetrically 
around the mean values. 
Consequently, the PC expansions include multivariate basis polynomials, 
which are products of univariate Legendre polynomials. 
Table~\ref{tab:number-basis} depicts the number~$s$ of basis polynomials 
for different total degrees~$d$ in a truncated 
PC expansion~\eqref{eq:pc-truncated}.

We consider both transformations of Section~\ref{sec:trafo1} 
and Section~\ref{sec:trafo2}, respectively. 
In the first transformation, 
the matrix square root~\eqref{eq:motor-square-root} is used 
in the symmetric decomposition. 
Furthermore, we examine different polynomial degrees up to six. 
In each case, we arrange three SG systems: 
for the original system~\eqref{eq:ph} and for the two 
transformed systems~\eqref{eq:ph-trafo-tilde}. 
Hence the first SG system is not in pH form. 
The matrices of the SG systems are always computed 
using a tensor-product Gauss-Legendre quadrature with $7^5 = 16807$ nodes 
to guarantee a high accuracy. 
All computed SG systems are asymptotically stable. 

%%%%%%%%%%%%%%%%%%%%%%%%%%%%%%%%%%%%%%%%%%%%%%%%%%%%%%%%%%%%%%%%%%%%%%%%%%%%%%
\begin{table}
    \centering
    \caption{Number of basis polynomials for different total degrees.}
    \begin{tabular}{lcccccc}
         degree~$d$ & 1 & 2 & 3 & 4 & 5 & 6 \\ \hline
         number~$s$ & 6 & 21 & 56 & 126 & 252 & 462 
    \end{tabular}
    \label{tab:number-basis}
\end{table}
%%%%%%%%%%%%%%%%%%%%%%%%%%%%%%%%%%%%%%%%%%%%%%%%%%%%%%%%%%%%%%%%%%%%%%%%%%%%%%

We compare the three SG systems using the $\htwonorm$-norm, 
which is the norm of an associated Hardy space, see~\cite{antoulas}.
Let $H_0,H_1,H_2$ be the transfer functions of the SG systems from 
the original system and the first/second transformation, respectively.  
We observe the relative differences 
\begin{equation} \label{eq:rel-diff}
D_{\rm rel} = 
\frac{\| H_0 - H_i \|_{\htwonorm}}{\| H_0 \|_{\htwonorm}} 
\qquad \mbox{for} \;\; i=1,2 . 
\end{equation}
The SG systems exhibit multiple-input-multiple-output (MIMO). 
We also investigate the restriction to single-input-single-output (SISO), 
where just the stochastic mode associated to degree zero is considered 
for inputs and outputs, 
and to single-input-multiple-output (SIMO), 
where the restriction applies only to the inputs.

We study the two cases of 1\% and 10\% variation around the 
mean values of the parameters. 
Figure~\ref{fig:galerkin-differences1} and 
Figure~\ref{fig:galerkin-differences2} illustrate the relative 
differences~(\ref{eq:rel-diff}) with respect to the total degree. 
Each SG method is convergent separately. 
Thus the differences tend to zero for increasing degree 
in the case of SISO and SIMO. 
The differences do not converge to zero in the case of MIMO. 
If an input $u(t,\mu)$ is in the space~$\ltworandom$ for fixed~$t$, 
then its PC coefficients converge to zero. 
The system norm does not take this behaviour into account. 
Furthermore, larger variations of the random parameters slow down 
the convergence. 

%%% Figure %%%%%%%%%%%%%%%%%%%%%%%%%%%%%%%%%%%%%%%%%%%%%%%%%%%%%%%%%%%%%%%%%%%
\begin{figure}
  \centering
  \begin{subfigure}[b]{0.45\textwidth}
  \centering
  \includegraphics[width=\textwidth]{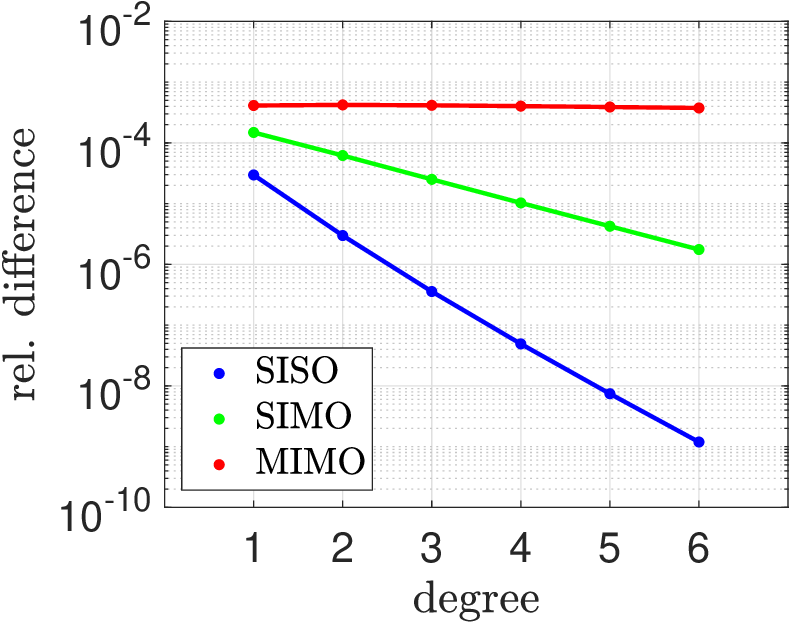}
  \caption{transformation, Section~\ref{sec:trafo1}}
  \end{subfigure}
  \hspace{8mm}
  \begin{subfigure}[b]{0.45\textwidth}
  \centering
  \includegraphics[width=\textwidth]{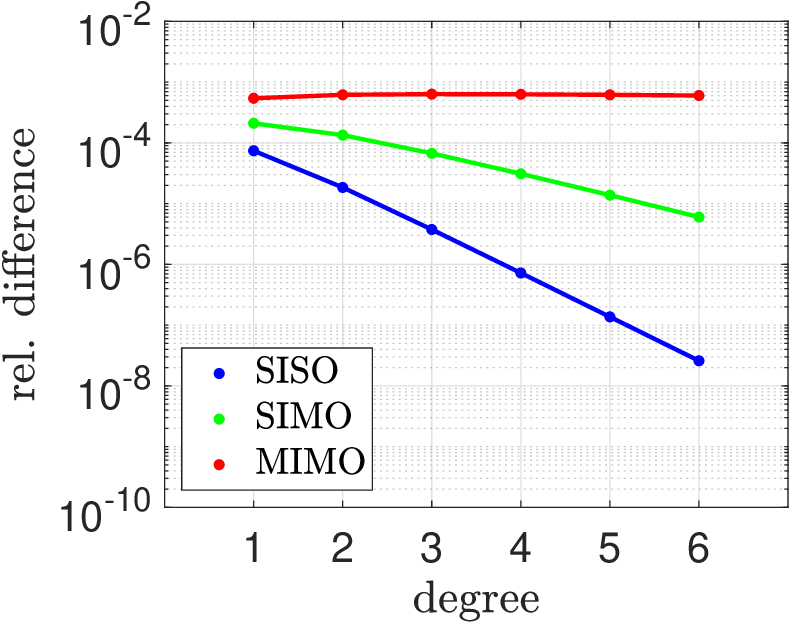}
  \caption{transformation, Section~\ref{sec:trafo2}}
  \end{subfigure}
  \caption{Relative differences in $\mathcal{H}_2$-norm between 
  SG systems from original system and transformed system
  in the case of 1\% variation.}
\label{fig:galerkin-differences1}
\end{figure}
%%%%%%%%%%%%%%%%%%%%%%%%%%%%%%%%%%%%%%%%%%%%%%%%%%%%%%%%%%%%%%%%%%%%%%%%%%%%%

%%% Figure %%%%%%%%%%%%%%%%%%%%%%%%%%%%%%%%%%%%%%%%%%%%%%%%%%%%%%%%%%%%%%%%%%%
\begin{figure}
  \centering
  \begin{subfigure}[b]{0.45\textwidth}
  \centering
  \includegraphics[width=\textwidth]{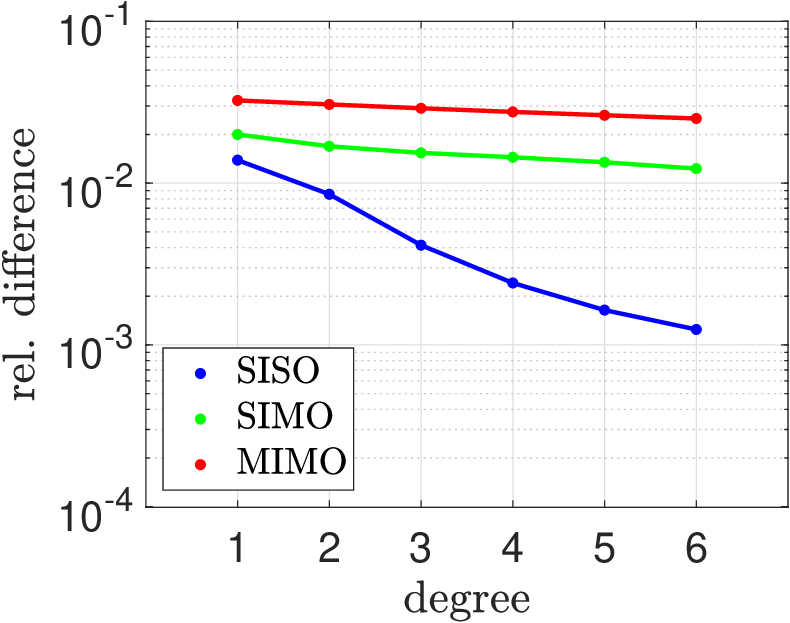}
  \caption{transformation, Section~\ref{sec:trafo1}}
  \end{subfigure}
  \hspace{8mm}
  \begin{subfigure}[b]{0.45\textwidth}
  \centering
  \includegraphics[width=\textwidth]{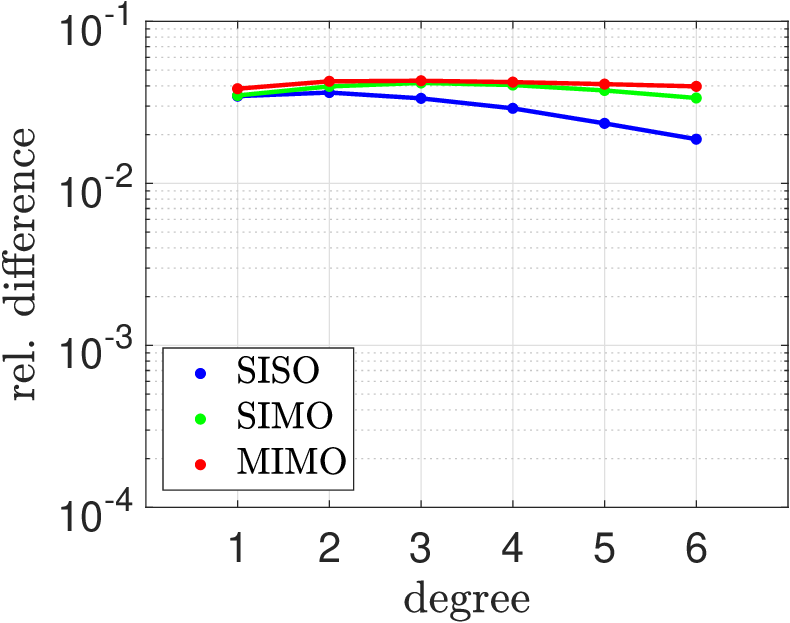}
  \caption{transformation, Section~\ref{sec:trafo2}}
  \end{subfigure}
  \caption{Relative differences in $\mathcal{H}_2$-norm between 
  SG systems from original system and transformed system
  in the case of 10\% variation.}
\label{fig:galerkin-differences2}
\end{figure}
%%%%%%%%%%%%%%%%%%%%%%%%%%%%%%%%%%%%%%%%%%%%%%%%%%%%%%%%%%%%%%%%%%%%%%%%%%%%%

Moreover, we execute a transient simulation of the systems for 1\% variation 
in the time interval $[0,300]$. 
Initial values are always zero. 
As input, we employ the signal $u(t) = \sin(t^2)$, 
which can be interpreted as a harmonic oscillation with increasing frequency. 
Initial value problems (IVPs) are solved by an explicit Runge-Kutta method 
of order 4(5) including variable step sizes based on a local error control. 
We consider the two SG-pH systems from the transformations 
with total degree four in the SIMO case. 
Figure~\ref{fig:motor-statistics} shows the expected value as well as the 
standard deviation of the random output current 
in the system~\eqref{eq:dcmotor}, 
which are obtained by the solution of the first SG system. 
Although the variation of the random parameters is low (1\%), 
the standard deviation of the random output is relatively large 
near the resonance frequencies.
The Hamiltonian function of the first SG system is displayed 
by Figure~\ref{fig:motor-hamiltonian}. 
We resolve the IVPs in the time interval $[0,200]$ 
using high accuracy in the time integration. 
Figure~\ref{fig:motor-error} (a) illustrates the (absolute) differences 
between the Hamiltonian functions of the two SG systems, 
which are tiny. 

For comparison, we compute an approximation of the expected value of the
random Hamiltonian function associated to the 
original system~\eqref{eq:dcmotor}. 
A tensor-prod\-uct Gauss-Legendre quadrature is used with 
$3^5 = 243$ nodes. 
Thus 234 IVPs of the system~\eqref{eq:dcmotor} are solved for 
different realisations of the parameters imposing a high accuracy. 
Figure~\ref{fig:motor-error} (b) depicts the (absolute) differences 
between the Hamiltonian function of the first SG system and the 
approximation of the expected value. 
Since the differences are very small, 
the statement of Theorem~\ref{thm:hamiltonian} is confirmed in this example. 

%%% Figure %%%%%%%%%%%%%%%%%%%%%%%%%%%%%%%%%%%%%%%%%%%%%%%%%%%%%%%%%%%%%%%%%%%
\begin{figure}
  \centering
  \begin{subfigure}[b]{0.45\textwidth}
  \centering
  \includegraphics[width=\textwidth]{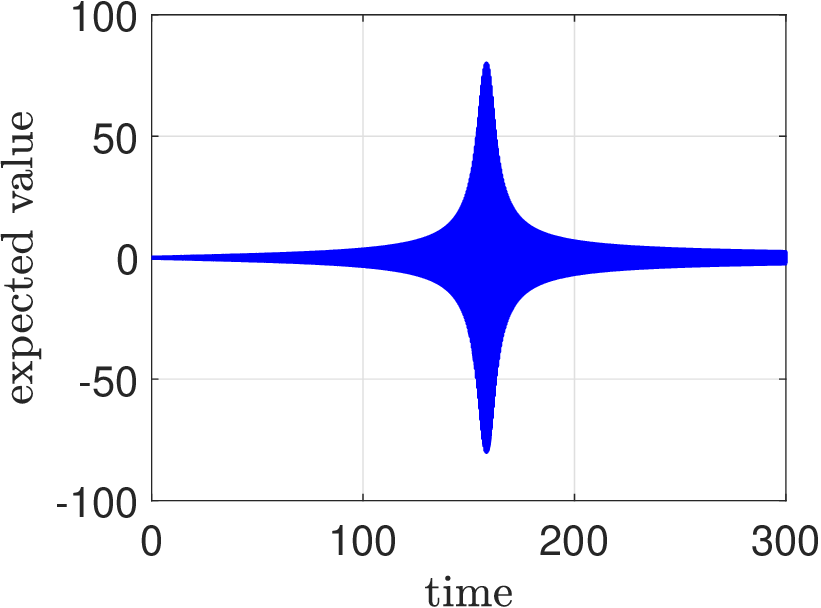}
  \caption{expected value}
  \end{subfigure}
  \hspace{8mm}
  \begin{subfigure}[b]{0.45\textwidth}
  \centering
  \includegraphics[width=\textwidth]{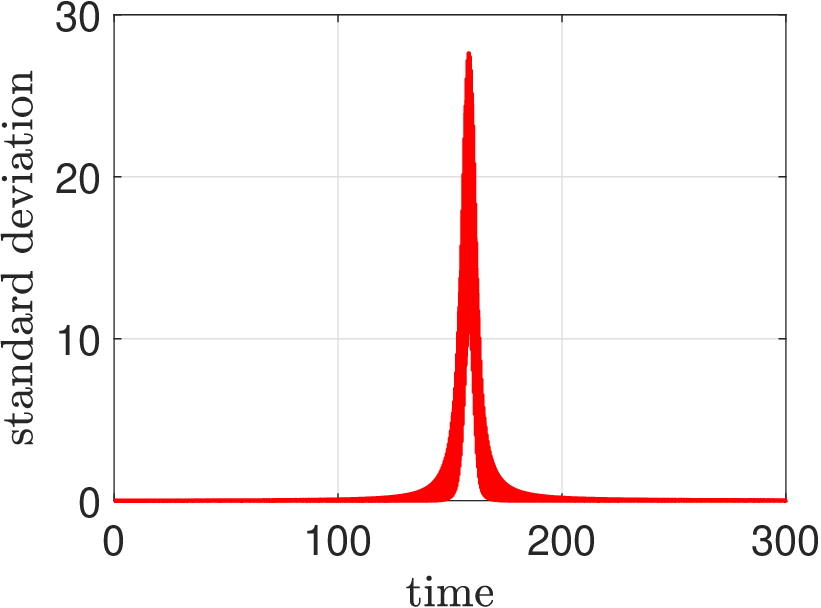}
  \caption{standard deviation}
  \end{subfigure}
  \caption{Expected value and standard deviation of random output 
  in example of DC motor.}
\label{fig:motor-statistics}
\end{figure}
%%%%%%%%%%%%%%%%%%%%%%%%%%%%%%%%%%%%%%%%%%%%%%%%%%%%%%%%%%%%%%%%%%%%%%%%%%%%%

%%% Figure %%%%%%%%%%%%%%%%%%%%%%%%%%%%%%%%%%%%%%%%%%%%%%%%%%%%%%%%%%%%%%%%%%%
\begin{figure}
  \centering
  \begin{subfigure}[b]{0.45\textwidth}
  \centering
  \includegraphics[width=\textwidth]{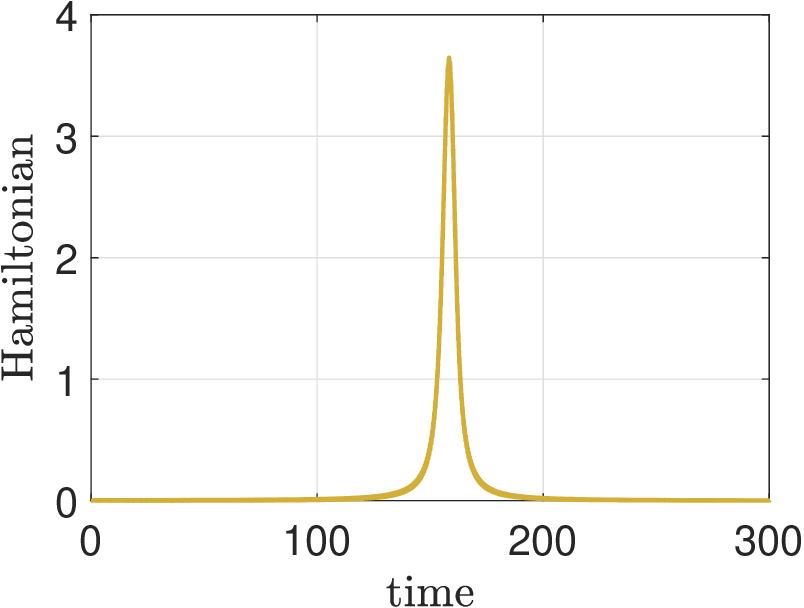}
  \caption{total time interval}
  \end{subfigure}
  \hspace{8mm}
  \begin{subfigure}[b]{0.45\textwidth}
  \centering
  \includegraphics[width=\textwidth]{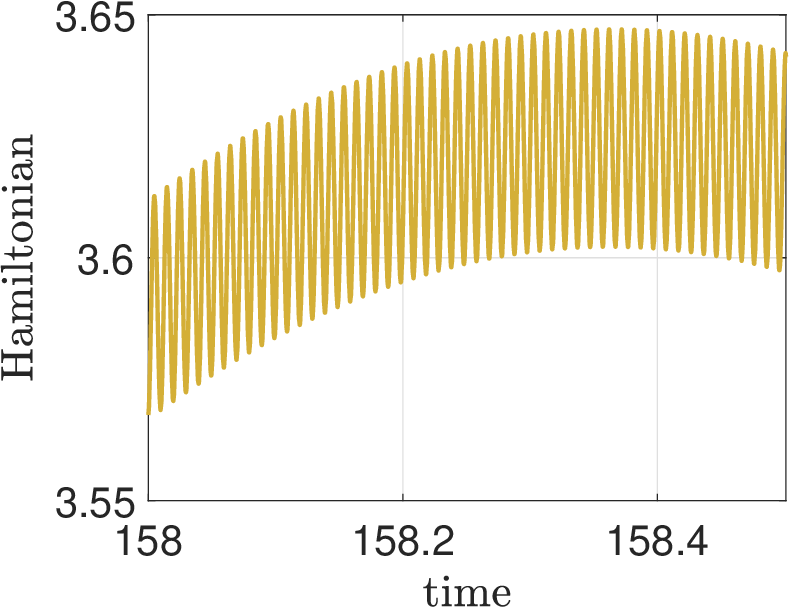}
  \caption{zoom}
  \end{subfigure}
  \caption{Hamiltonian function of SG system in example of DC motor.}
\label{fig:motor-hamiltonian}
\end{figure}
%%%%%%%%%%%%%%%%%%%%%%%%%%%%%%%%%%%%%%%%%%%%%%%%%%%%%%%%%%%%%%%%%%%%%%%%%%%%%

%%% Figure %%%%%%%%%%%%%%%%%%%%%%%%%%%%%%%%%%%%%%%%%%%%%%%%%%%%%%%%%%%%%%%%%%%
\begin{figure}
  \centering
  \begin{subfigure}[b]{0.45\textwidth}
  \centering
  \includegraphics[width=\textwidth]{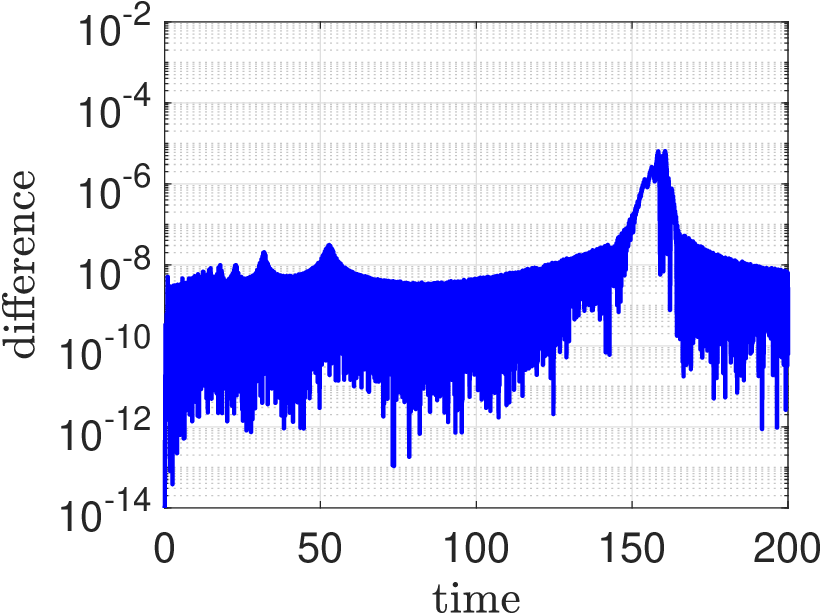}
  \caption{SG systems}
  \end{subfigure}
  \hspace{8mm}
  \begin{subfigure}[b]{0.45\textwidth}
  \centering
  \includegraphics[width=\textwidth]{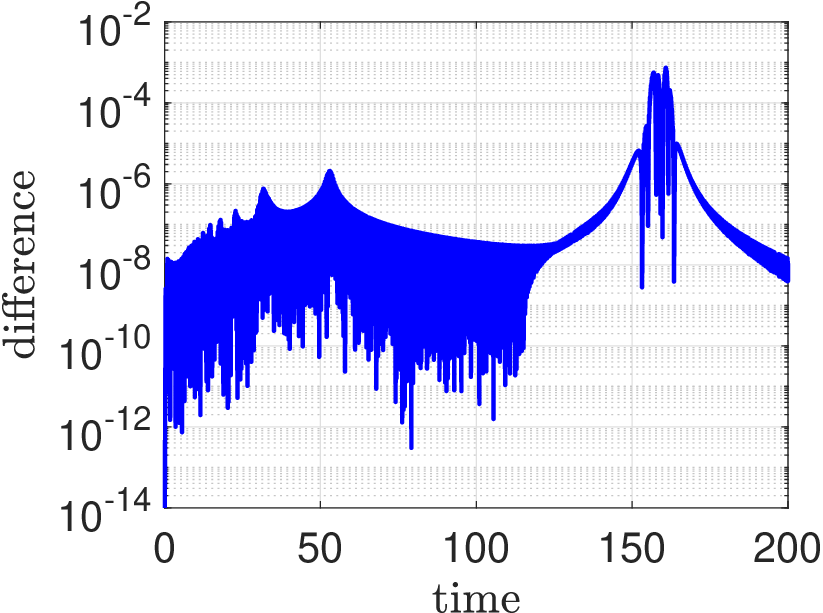}
  \caption{SG system and quadrature}
  \end{subfigure}
  \caption{Differences between Hamiltonian functions in example of DC motor.}
\label{fig:motor-error}
\end{figure}
%%%%%%%%%%%%%%%%%%%%%%%%%%%%%%%%%%%%%%%%%%%%%%%%%%%%%%%%%%%%%%%%%%%%%%%%%%%%%

We do not apply MOR in this test example, 
because the dimensions of the SG systems are relatively low. 
In contrast, high-dimensional SG systems are 
obtained in the following test example. 

%%%%%%%%%%%%%%%%%%%%%%%%%%%%%%%%%%%%%%%%%%%%%%%%%%%%%%%%%%%%%%%%%%%%%%%%%%%%%
\subsection{RLC Ladder Network}
\label{sec:network}
In~\cite{ionescu-astolfi}, an electric ladder network was considered as 
test example, which implies a system of four ODEs. 
We extend this example to $k$ cells, 
where the dimension of the ODE system becomes $n=2k$. 
The electric network includes $k$~capacitances~$C_i$, 
$k$~inductances~$L_i$, and $k$~resistances $R_i$. 
The state variables represent charges and fluxes 
$x = (q_1,\phi_1,\ldots,q_k,\phi_k)^\top$. 
The single input is an input current $u=I$, 
whereas the single output is the charge at the first capacitance $y=q_1$. 
The pH form features the tridiagonal skew-symmetric matrix
\begin{equation*}
J = 
\begin{pmatrix}
0 & -1 & \cdots & 0 \\
1 & 0 & \ddots & \vdots \\
\vdots & \ddots & \ddots & -1 \\
0 & \cdots & 1 & 0 
\end{pmatrix},
\end{equation*}
the two diagonal matrices
\begin{equation*}
R = \diag(0,R_1,0,R_2,\ldots,0,R_k) \quad \mbox{and} \quad
Q = \diag(\tfrac{1}{C_1},\tfrac{1}{L_1},\ldots,
\tfrac{1}{C_k},\tfrac{1}{L_k}) ,
\end{equation*}
and the input matrix $B = (1,0,\ldots,0)^\top$. 
We rearrange the parameters into 
$\mu_i = \frac{1}{C_i}$, $\mu_{k+i} = \frac{1}{L_i}$, 
$\mu_{2k+i} = R_i$ for $i=1,\ldots,k$. 
Now both $R(\mu)$ and $Q(\mu)$ are affine-linear functions 
of the parameters~$\mu$. 

%%%%%%%%%%%%%%%%%%%%%%%%%%%%%%%%%%%%%%%%%%%%%%%%%%%%%%%%%%%%%%%%%%%%%%%%%%%%%
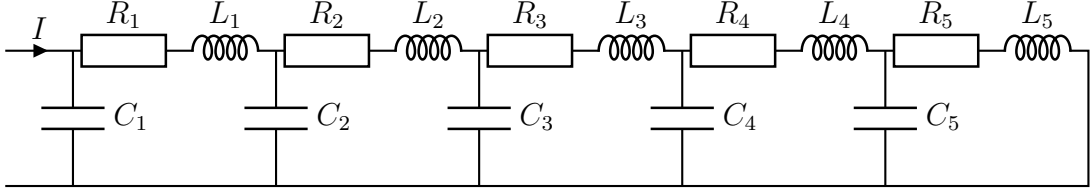
\begin{figure}
{\centering
\begin{circuitikz}[scale=0.9]
\draw[line width=0.8]
(-1,0) to [short, i=$I$] (0,0)
to [european resistor, l=$R_1$] ++ (1.5,0) to [L, l=$L_1$] ++ (1.5,0)
to [european resistor, l=$R_2$] ++ (1.5,0) to [L, l=$L_2$] ++ (1.5,0)
to [european resistor, l=$R_3$] ++ (1.5,0) to [L, l=$L_3$] ++ (1.5,0)
to [european resistor, l=$R_4$] ++ (1.5,0) to [L, l=$L_4$] ++ (1.5,0)
to [european resistor, l=$R_5$] ++ (1.5,0) to [L, l=$L_5$] ++ (1.5,0)
to ++(0,-2) to (-1, -2);
;
% Capacitors
\draw[line width=0.8] (0,0) to [C, l=$C_1$] + (0,-2);
\draw[line width=0.8] (3,0) to [capacitor, l=$C_2$] + (0,-2);
\draw[line width=0.8] (6,0) to [capacitor, l=$C_3$] + (0,-2);
\draw[line width=0.8] (9,0) to [capacitor, l=$C_4$] + (0,-2);
\draw[line width=0.8] (12,0) to [capacitor, l=$C_5$] + (0,-2);
\end{circuitikz}

}
\caption{RLC ladder network, see Section~\ref{sec:network}.}
\label{fig:rlc-ladder-network}
\end{figure}
%%%%%%%%%%%%%%%%%%%%%%%%%%%%%%%%%%%%%%%%%%%%%%%%%%%%%%%%%%%%%%%%%%%%%%%%%%%%%

%%% Figure: Bode Plot %%%%%%%%%%%%%%%%%%%%%%%%%%%%%%%%%%%%%%%%%%%%%%%%%%%%%%%
\begin{figure}
  \centering
  \begin{subfigure}[b]{0.45\textwidth}
  \centering
  \includegraphics[width=\textwidth]{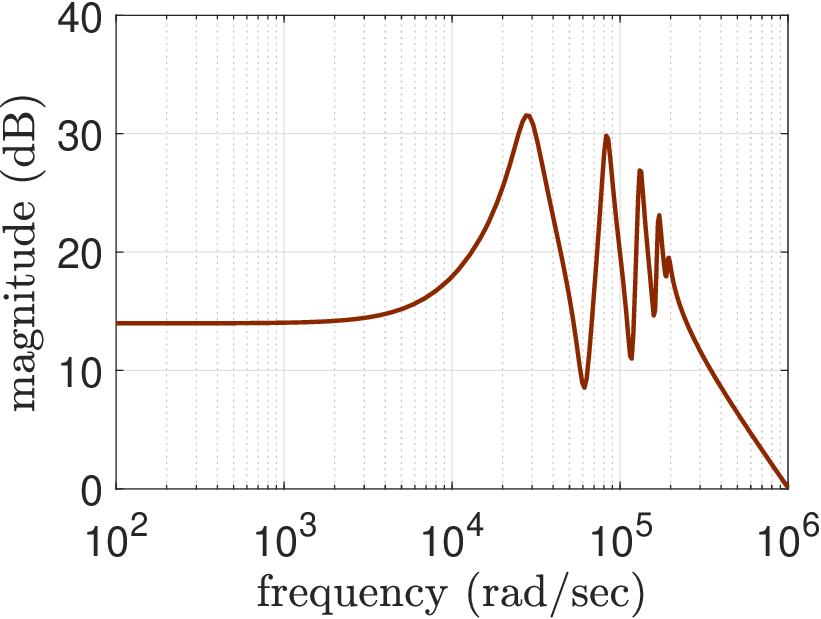}
  \caption{magnitude}
  \end{subfigure}
  \hspace{8mm}
  \begin{subfigure}[b]{0.45\textwidth}
  \centering
  \includegraphics[width=\textwidth]{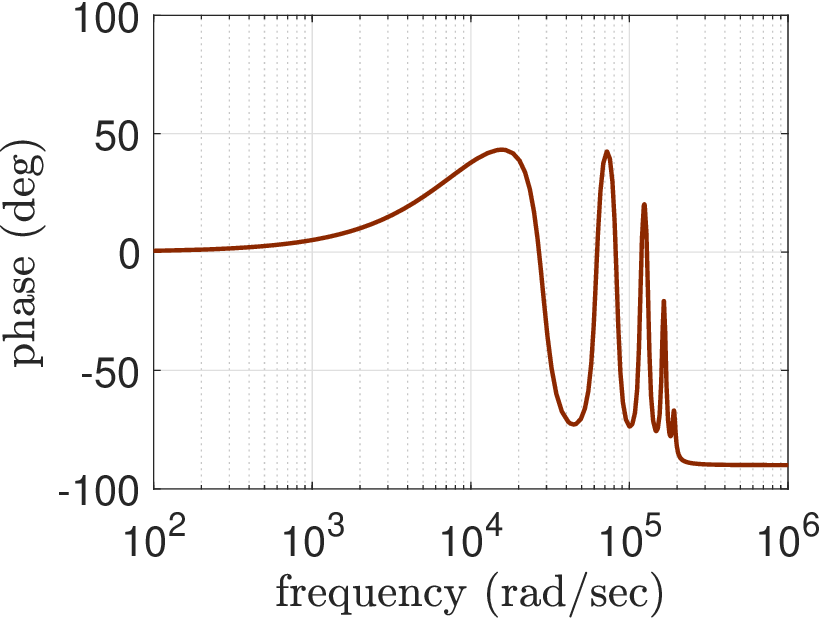}
  \caption{phase}
  \end{subfigure}
  \caption{Bode plot of RLC ladder network for deterministic 
  physical parameters.}
\label{fig:bode-rlc-ladder}
\end{figure}
%%%%%%%%%%%%%%%%%%%%%%%%%%%%%%%%%%%%%%%%%%%%%%%%%%%%%%%%%%%%%%%%%%%%%%%%%%%%%

In the following, we always use $k=5$ cells, which implies a system 
of $n=10$ ODEs including $q=15$ parameters. 
Figure~\ref{fig:rlc-ladder-network} illustrates this electric circuit. 
We choose the constant parameters 
$\bar{C} = 10^{-6}$ for all capacitances, 
$\bar{L} = 10^{-4}$ for all inductances, and 
$\bar{R} = 1$ for all resistances. 
The Bode plot of the resulting linear dynamical system is shown in
Figure~\ref{fig:bode-rlc-ladder}. 

Now we use a transformation to obtain a pH system of the 
form~\eqref{eq:ph-trafo} with $\tilde{Q} = I_{n}$. 
The transformation of Section~\ref{sec:trafo1} is not employed, 
because the resulting system matrices are not polynomials any more. 
Alternatively, we use the transformation of Section~\ref{sec:trafo2}. 
It follows that $\tilde{J}(\mu)$ includes quadratic polynomials, 
$\tilde{R}(\mu)$ consists of cubic polynomials, 
and $\tilde{B}(\mu),\tilde{E}(\mu)$ contain linear polynomials. 

In the stochastic modelling, we change the parameters~$\mu$  
into independent random variables with uniform probability distributions. 
The mean values are defined as the above constant choice of the parameters, 
whereas each random variable varies $10\%$ around its mean value. 
We consider two SG systems associated to the 
total degrees two and three. 
Since the matrices of the transformed pH system~\eqref{eq:ph-trafo} 
are polynomials of the random variables, 
the matrices of the SG system are computed 
semi-analytically with an accuracy up to machine precision. 
Table~\ref{tab:rlc-ladder} shows the properties of the 
SG systems. 
In particular, the ratio of non-zero elements in the sparse matrices 
is specified. 
The sparsity patterns of the system matrices are displayed 
in the case of degree three by Figure~\ref{fig:rlc-ladder-matrices}. 
Furthermore, both SG systems are asymptotically stable. 
%Table~\ref{tab:rlc-ladder-norms} depicts the $\htwonorm$-norms of 
%the SG systems, where also restrictions to a single input 
%or a single output are addressed.

%%%%%%%%%%%%%%%%%%%%%%%%%%%%%%%%%%%%%%%%%%%%%%%%%%%%%%%%%%%%%%%%%%%%%%%%%%%%%%
\begin{table}
    \centering
    \caption{Properties of SG systems in the example of the
    RLC ladder network.}
    \begin{tabular}{cccccc}
         & size of & system & non-zeros & 
         non-zeros & non-zeros \\
         degree & basis & dimension & in $\hat{J}$ & 
         in $\hat{R}$ & in $\hat{Q}$ \\ \hline
         %2 & 1360 & $2.3 \cdot 10^{-3}$ & $6.4 \cdot 10^{-4}$ & 
         %$9.1 \cdot 10^{-4}$ \\
         %3 & 8160 & $4.5 \cdot 10^{-4}$ & $1.3 \cdot 10^{-4}$ & 
         %$1.6 \cdot 10^{-4}$ 
         2 & 136 & 1360 & 0.228\% & 0.064\% & 0.091\% \\
         3 & 816 & 8160 & 0.045\% & 0.013\% & 0.016\%
    \end{tabular}
    \label{tab:rlc-ladder}
\end{table}
%%%%%%%%%%%%%%%%%%%%%%%%%%%%%%%%%%%%%%%%%%%%%%%%%%%%%%%%%%%%%%%%%%%%%%%%%%%%%%

%%%%%%%%%%%%%%%%%%%%%%%%%%%%%%%%%%%%%%%%%%%%%%%%%%%%%%%%%%%%%%%%%%%%%%%%%%%%%%
%\begin{table}
%    \centering
%    \caption{$\htwonorm$-norms of SG systems in the example of the
%    RLC ladder network.}
%    \begin{tabular}{cccc}
%         degree & SISO & SIMO & MIMO \\ \hline
%         2 & $3.85 \cdot 10^3$ & $3.95 \cdot 10^3$ & $4.60 \cdot 10^4$ \\
%         3 & $3.85 \cdot 10^3$ & $3.95 \cdot 10^3$ & $1.13 \cdot 10^6$  
%    \end{tabular}
%    \label{tab:rlc-ladder-norms}
%\end{table}
% MIMO - degree 3: 112699.56351 
%%%%%%%%%%%%%%%%%%%%%%%%%%%%%%%%%%%%%%%%%%%%%%%%%%%%%%%%%%%%%%%%%%%%%%%%%%%%%%

%%% Figure: Matrices %%%%%%%%%%%%%%%%%%%%%%%%%%%%%%%%%%%%%%%%%%%%%%%%%%%%%%%%%
\begin{figure}
  \centering
  \includegraphics[width=0.8\textwidth]{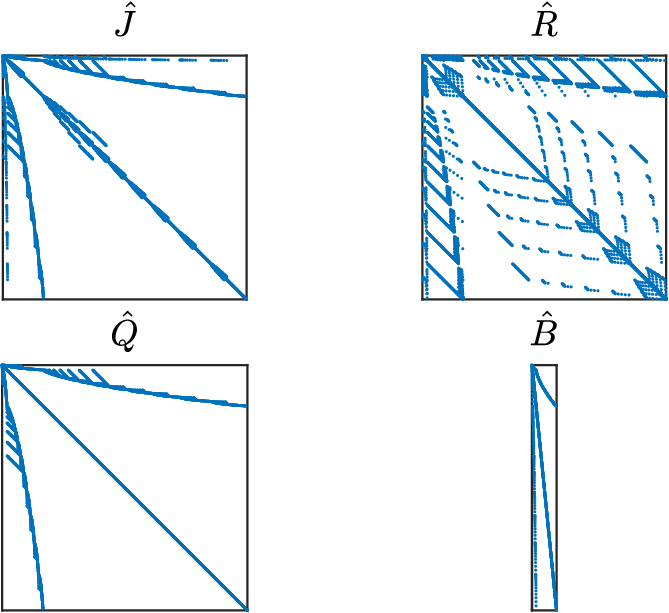}
  \caption{Sparsity patterns of matrices in the SG system 
  for polynomial of degree~$3$ in the example of the RLC ladder network.}
\label{fig:rlc-ladder-matrices}
\end{figure}
%%%%%%%%%%%%%%%%%%%%%%%%%%%%%%%%%%%%%%%%%%%%%%%%%%%%%%%%%%%%%%%%%%%%%%%%%%%%%

Now we study MOR applied to the SG systems. 
Since the input current does not depend on the physical parameters, 
we employ a single input in the following. 
Thus the input matrix consists of a single column. 
Three methods are used:
\begin{itemize}
    \item[i)] (one-sided) Arnoldi algorithm, see~\cite{antoulas},
    \item[ii)] an iterative rational Krylov algorithm (IRKA) 
          as proposed in~\cite{gugercin2008},
    \item[iii)] balanced truncation technique, see~\cite{antoulas}.
\end{itemize}
The methods (i) and (ii) represent Krylov subspace techniques. 
The Arnoldi algorithm is a Galerkin-type MOR method, 
which does not utilise information on the definition of outputs in the system. 
The used IRKA technique is a Petrov-Galerkin-type MOR method ($W \neq V$). 
However, we ignore the projection matrix~$W$ and set $W := V$ 
to obtain a Galerkin-type method. 
Furthermore, the method~(ii) is applied to the associated SISO systems 
including a high accuracy requirement. 
% tolerance $10^{-12}$.
We use the toolbox sssMOR in MATLAB, see~\cite{castagnotto}, 
to compute the projection matrices in the Arnoldi algorithm 
as well as the IRKA.
For comparison, the balanced truncation technique~(iii) is employed 
to reduce the SIMO systems.
Therein, the projection matrices are calculated by direct methods of 
numerical linear algebra. 
The methods~(i) and~(ii) preserve the pH form of the 
SG systems due to Theorem~\ref{thm:galerkin-mor}. 
However, the method~(iii) does not yield reduced systems in pH form 
but in the general form~(\ref{eq:ode}).

We use the MOR methods to determine the projection matrices for 
reduced dimension~$60$. 
Table~\ref{tab:computing-times} shows the required computation times. 
We use the columns of the matrices to obtain the ROMs of 
dimensions $r=5,6,\ldots,60$.
In the IRKA, the projection matrix should be computed for each~$r$ 
separately. 
However, this method fails for low reduced dimensions, 
since some complex numbers cannot be paired. 
All calculated ROMs are asymptotically stable linear dynamical systems. 
Now we investigate the accuracy of the MOR methods. 
Let $H_{\rm FOM}$ and $H_{\rm ROM}$ be the transfer functions of 
an FOM and an ROM, respectively, in SIMO form. 
Hence the errors take all outputs into account. 
We measure relative errors by 
\begin{equation} \label{eq:mor-error} 
E_{\rm rel} = 
\frac{\| H_{\rm FOM} - H_{\rm ROM} \|_{\htwonorm}}{\| H_{\rm FOM} \|_{\htwonorm}} 
\end{equation}
using the $\htwonorm$-norm.
Figure~\ref{fig:errors-mor} demonstrates the errors~\eqref{eq:mor-error} 
of the three MOR approaches. 
The errors decay rapidly for increasing reduced dimensions in each method. 
As expected, the Arnoldi algorithm mostly exhibits the lowest accuracy, 
while the balanced truncation technique yields the highest accuracy. 
Nevertheless, the IRKA generates ROMs with errors of a magnitude 
close to the balanced truncation scheme for dimensions $50 \le r \le 60$.

%%% Table: Computation Times %%%%%%%%%%%%%%%%%%%%%%%%%%%%%%%%%%%%%%%%%%%%%%%%
\begin{table}
    \centering
    \caption{Computation times (in seconds) for projection matrix of 
    dimension~60 in MOR methods for example of RLC ladder network.}
    \begin{tabular}{cccc}
         degree & Arnoldi & IRKA & Bal. Tr. \\ \hline
         2 & 0.3 & 2.7 & 7.3 \\
         3 & 9.3 & 56.7 & 1892.8  
    \end{tabular}
    \label{tab:computing-times}
\end{table}
%%%%%%%%%%%%%%%%%%%%%%%%%%%%%%%%%%%%%%%%%%%%%%%%%%%%%%%%%%%%%%%%%%%%%%%%%%%%%%

%%% Figure: MOR Errors %%%%%%%%%%%%%%%%%%%%%%%%%%%%%%%%%%%%%%%%%%%%%%%%%%%%%%%
\begin{figure}
  \centering
  \begin{subfigure}[b]{0.45\textwidth}
  \centering
  \includegraphics[width=\textwidth]{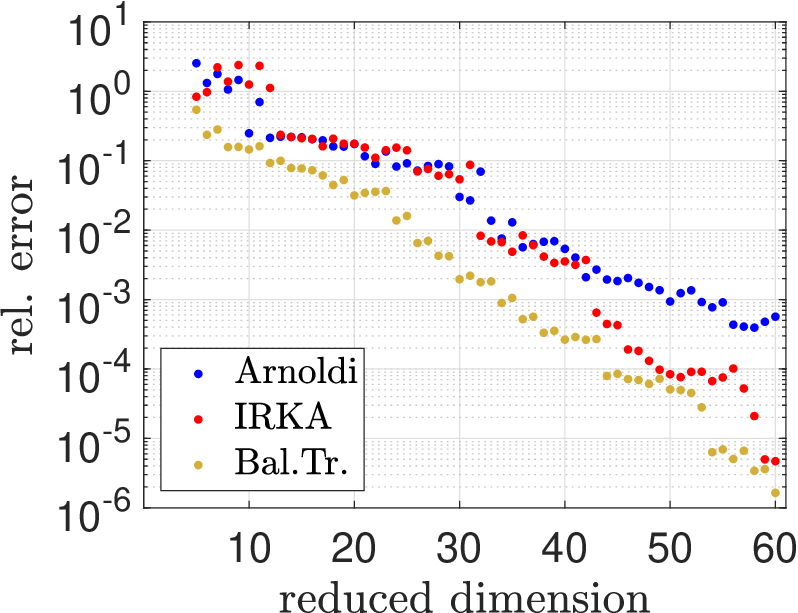}
  \caption{degree~2}
  \end{subfigure}
  \hspace{8mm}
  \begin{subfigure}[b]{0.45\textwidth}
  \centering
  \includegraphics[width=\textwidth]{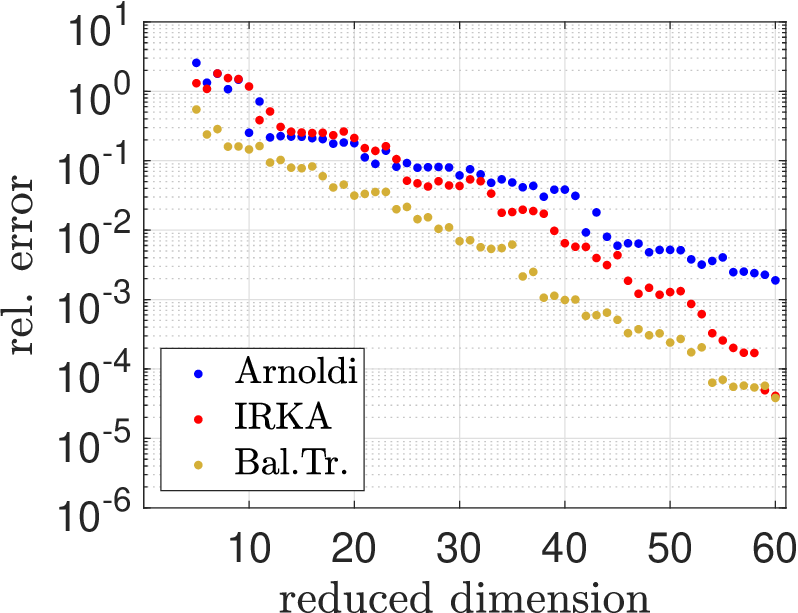}
  \caption{degree~3}
  \end{subfigure}
  \caption{Relative errors~\eqref{eq:mor-error} of MOR for SG systems 
  using different methods in example of RLC ladder network.}
\label{fig:errors-mor}
\end{figure}
%%%%%%%%%%%%%%%%%%%%%%%%%%%%%%%%%%%%%%%%%%%%%%%%%%%%%%%%%%%%%%%%%%%%%%%%%%%%%

%%%%%%%%%%%%%%%%%%%%%%%%%%%%%%%%%%%%%%%%%%%%%%%%%%%%%%%%%%%%%%%%%%%%%%%%%%%%%
%%%                          Conclusions                                  %%%
%%%%%%%%%%%%%%%%%%%%%%%%%%%%%%%%%%%%%%%%%%%%%%%%%%%%%%%%%%%%%%%%%%%%%%%%%%%%%

\section{Conclusions}
We considered linear first-order ODEs in pH form, 
where random variables are included. 
A SG approach preserves the pH structure, 
if an equivalent transformed system of a special form is used. 
We showed that the Hamiltonian function of an SG system represents 
an approximation of the expected value of the Hamiltonian function 
associated to the original random-dependent ODEs. 
Moreover, MOR methods of Galerkin-type are structure-preserving 
when applied to the SG systems in this pH form. 
Results of numerical computations confirmed the theoretical findings. 
In a test example, we demonstrated that typical Galerkin-type MOR techniques 
already yield reduced systems with a high accuracy.

%%%%%%%%%%%%%%%%%%%%%%%%%%%%%%%%%%%%%%%%%%%%%%%%%%%%%%%%%%%%%%%%%%%%%%%%%%%%%
%%%                           References                                  %%%
%%%%%%%%%%%%%%%%%%%%%%%%%%%%%%%%%%%%%%%%%%%%%%%%%%%%%%%%%%%%%%%%%%%%%%%%%%%%%

\end{document}